\documentclass[]{article}
\usepackage{indentfirst}
\usepackage{amscd}
\usepackage{amsmath}
\usepackage{latexsym}
\usepackage{amsfonts}
\usepackage{amssymb}
\usepackage{amsthm}
\usepackage{graphicx}
\usepackage{extarrows}
\usepackage{verbatim}
\usepackage{enumerate}
\usepackage{ulem}
\usepackage[CJKbookmarks=true]{hyperref}
\usepackage{xcolor}
\usepackage{color}
\usepackage{cases}

\theoremstyle{plain}\newtheorem{definition}{Definition}[section]
\theoremstyle{definition}\newtheorem{theorem}{Theorem}[section]
\theoremstyle{plain}\newtheorem{lemma}[theorem]{Lemma}
\theoremstyle{plain}
\theoremstyle{plain}\newtheorem{proposition}[theorem]{Proposition}
\theoremstyle{remark}\newtheorem{remark}{Remark}[section]

\newcommand{\Div}{\mathrm{div}}
\newcommand{\R}{\mathbb{R}}
\newcommand{\be}{\begin{equation}}
\newcommand{\ee}{\end{equation}}
\newcommand{\ba}{\begin{aligned}}
\newcommand{\ea}{\end{aligned}}
\newcommand{\f}{\frac}

\newcommand{\ben}{\begin{enumerate}}
\newcommand{\een}{\end{enumerate}}

\newcommand{\ti}{\nabla}

\newcommand{\nt}{\notag}

\newcommand{\BF}{\textbf}

\begin{document}
\title{Existence and stability of stationary solutions to the compressible quantum model \hspace{-4mm}}
\author{ Xiaoyu Xi$^{\thanks{Corresponding author.  E-mail addresses: \it  sxxixiaoyu@gzhu.edu.cn(X. Xi).}}$\\[10pt]
\small { School  of Mathematics and Information Science, Guangzhou University,}\\
\small {Guangzhou, 510006, P. R. China}\\[5pt]
}



\date{}

\maketitle

\begin{abstract}In this paper, the compressible quantum model with the given mass source and the external force of general form in three-dimensional whole space is considered. Based on the weighted $L^2$ method and $L^\infty$ estimates, the existence and uniqueness of stationary solutions can be obtained by the contraction mapping principle. By using a general energy method, the nonlinear stability of stationary solutions is studied, and the time decay rates of the solutions are established when the initial perturbation belongs to $\dot{H}^{-s}$ with $0 \leq s < \f32$.

\vspace*{5pt}

\vspace*{5pt}
\end{abstract}


\section{Introduction}
The compressible viscous quantum model with the given mass source and the external force of general form can be written as
\begin{equation}\label{1.1}
		\left\{ \begin{aligned}
			&\rho _t+\mathrm{div}(\rho u)=G(x),\\
 			&(\rho u)_t+\mathrm{div}(\rho u\otimes u)-\mu\Delta u-(\mu+\lambda)\ti\Div u+\ti P-\f{\hbar^2}{2}\rho\ti \bigg(\f{\triangle\sqrt{\rho}}{\sqrt{\rho}}\bigg)=\rho F(x),\\
            &(\rho, u)(x,t)|_{t=0}=(\rho_0, u_0)(x),
		\end{aligned} \right.
	\end{equation}
where $(x, t) \in \mathbb{R}^3\times\mathbb{R}^{+}$ denotes the space-time position. The unknown functions $\rho$ and $u$ represent the density and velocity. $P=P(\rho)$ is the pressure, $\hbar$ denotes the Planck constant and $\nu$ is the magnetic diffusivity. The constants $\mu> 0$ and $\lambda$ are the viscosity coefficients with the usual physical condition $3\lambda+{2}\mu \geq 0$. $G(x)$ and $F(x)$ are the given mass source and the external force, respectively. The expression $\f{\triangle\sqrt{\rho}}{\sqrt{\rho}}$ can be interpreted as a quantum potential, i.e. Bohm potential, which satisfies
\begin{equation*}
2\rho\ti\bigg({\f{\triangle\sqrt{\rho}}{\sqrt{\rho}}}\bigg)=\Div(\rho\ti^2\mathrm{log}\rho)=\triangle\nabla\rho+\f{|\ti \rho|^2\ti \rho}{\rho^2}-\f{\ti \rho \triangle \rho}{\rho}-\f{\ti \rho\cdot \ti^2\rho}{\rho}.
\end{equation*}
The quantum correction terms, which are closely related to Bohm potential, can be tracked back to Wigner \cite{[28]} for the thermodynamic equilibrium.
For more physical interpretations of the model, one may refer to Hass \cite{[s7],[s8]}.
Furthermore, we assume
\begin{equation*}
\underset{|x|\rightarrow \infty}{\lim}(\rho_0(x)-\bar{\rho},{u}_0(x))=0,
\end{equation*}
where $\bar{\rho}$ is the constant.

The stationary problem corresponding to the problem \eqref{1.1} is
\begin{equation}\label{st1.1}
		\left\{ \begin{aligned}
			&\mathrm{div}(\rho u)=G(x),\\
 			&\mathrm{div}(\rho u\otimes u)-\mu\Delta u-(\mu+\lambda)\ti\Div u+\ti P-\f{\hbar^2}{2}\rho\ti \bigg(\f{\triangle\sqrt{\rho}}{\sqrt{\rho}}\bigg)=\rho F,
		\end{aligned} \right.
	\end{equation}
which satisfies
\begin{equation*}
\underset{|x|\rightarrow \infty}{\lim}(\rho(x)-\bar{\rho},{u}(x))=0.
\end{equation*}
$\mathbf{Notation.}$
Throughout this paper, the notation $a\sim b$ means that $C_1 b \leq a\leq C_2 b$. $C_1,~C_2$ are generic positive constants independent of $t$. We write $\int f dx=\int_{\mathbb{R}^3}f dx$ and use $L^p$ to denote the $L^p(\mathbb{R}^3)$ norm, for scalars $v_1$, $v_2$ and $u=(u_1,u_2,\cdots,u_n)$, $\omega=(\omega_1,\omega_2,\cdots,\omega_n)$, we put
\begin{eqnarray*}
&&\|u\|_k=\|u\|_{H^k},~~\|u\|=\|u\|_{L^2},\\
&&\langle v_1,v_2\rangle=\int_{\mathbb{R}^3}v_1v_2dx,~~\langle u,\omega\rangle=\sum\limits_{i=1}^n\langle u_i,\omega_i\rangle,\\
&&H^k=\{u\in L^1_{loc}\big|\|u\|_{H^k}<\infty\},~~\hat{H}^k=\{u\in L^1_{loc}\big|\|\ti u\|_{H^{k-1}}<\infty\},
	\end{eqnarray*}
where $u$ is either a vector or scalar. Furthermore, we put
\begin{eqnarray*}
&&\mathcal{H}^{k,l}=\{(\sigma,u)\big|\sigma\in H^k,u\in H^l\},\\
&&\hat{\mathcal{H}}^{k,l}=\{(\sigma,u)\big|\sigma\in \hat{H}^k,u\in \hat{H}^l\},\\
&&\|(\sigma,u)\|_{k,l}=\|\sigma\|_{k}+\|u\|_l.
	\end{eqnarray*}
Now, we define some function spaces which will be used later,
\begin{eqnarray*}
I^k_\epsilon=\{\sigma\in H^k\big|\|\sigma\|_{I^k}<\epsilon\},~~J^k_\epsilon=\{u\in H^k\big|\|u\|_{J^k}<\epsilon\},
	\end{eqnarray*}
where
\begin{align*}
\|\sigma\|_{I^k}&=\|\sigma\|_{L^6}+\sum\limits_{\nu=1}^{k}\|(1+|x|)^\nu(\ti^{\nu}\sigma,\ti^{\nu+1}\sigma,\ti^{\nu+2}\sigma)\|+\|(1+|x|)^2\sigma\|_{L^\infty}\\
&~~+\|(1+|x|)^2\ti\sigma\|_{L^\infty},\\
\|u\|_{J^k}&=\|u\|_{L^6}+\sum\limits_{\nu=1}^{k}\|(1+|x|)^{\nu-1}\ti^{\nu}u\|+\sum\limits_{\nu=0}^{1}\|(1+|x|)^{\nu+1}\ti^{\nu}u\|_{L^\infty}\\
&~~+\|(1+|x|)^2\ti^2u\|_{L^\infty}.
\end{align*}
Moreover, we set
\begin{eqnarray*}
&&\mathcal{F}^{j,k}_{\epsilon}=\{(\sigma,u)|\sigma\in I^j_\epsilon,u\in J^k_\epsilon,\|(\sigma,u)\|_{\mathcal{F}^{j,k}}<\epsilon\},\\
&&\|(\sigma,u)\|_{\mathcal{F}^{j,k}}=\|\sigma\|_{I^j}+\|u\|_{J^k},\\
&&\dot{\mathcal{F}}^{j,k}_{\epsilon}=\{(\sigma,u)\in\mathcal{F}^{j,k}_{\epsilon}|\Div u=\Div V_1+V_2,\|(1+|x|)^3V_1\|_{L^\infty}+\|(1+|x|)^{-1}V_2\|_{L^1}\leq\epsilon\},
	\end{eqnarray*}
and
\begin{eqnarray*}
    K_0&=&\|(1+|x|)G\|+\sum\limits_{\nu=0}^3\|(1+|x|)^{\nu+1}\ti^\nu F\|+\sum\limits_{\nu=1}^4\|(1+|x|)^{\nu}\ti^\nu G\|,\\
K&=&\|(1+|x|)G\|+\sum\limits_{\nu=0}^3\|(1+|x|)^{\nu+1}\ti^\nu F\|+\sum\limits_{\nu=1}^4\|(1+|x|)^{\nu}\ti^\nu G\|\\
    &&+\|(1+|x|)^3(F,G,\ti F,\ti G)\|_{L^\infty}+\|(1+|x|)^2F_1\|_{L^\infty}+\|F_2\|_{L^1}\\
    &&+\|(1+|x|)^3\ti^2 G\|_{L^\infty}.
\end{eqnarray*}
In this paper, the external force is given by the following form
\begin{eqnarray*}
F=\Div F_1+F_2.
\end{eqnarray*}
The three main results are as follows. First, for the existence of stationary solutions of the problem \eqref{st1.1} and its weighted $L^2$ and $L^\infty$ estimates, we state the following:
\begin{theorem}\label{thm1.1}
Let $P$ is a smooth solution in a neighborhood of $\bar{\rho}$. There exist small constants $c_0$ and $\epsilon_0$ such that if $(F,G)$ satisfies
\begin{eqnarray*}
    K+\|(1+|x|)^{-1}G\|_{L^1}\leq c_0\epsilon,
\end{eqnarray*}
for some $\epsilon\leq \epsilon_0$.
Then the problem \eqref{st1.1} admits a solution of the form: $(\rho,u)=(\sigma+\bar{\rho},u)$ where $(\sigma,u)\in \dot{\mathcal{F}}^{4,5}_{\epsilon}$. Moreover, the solution is unique in the following sense: if $(\sigma_1+\bar{\rho},u_1)$ and $(\sigma_2+\bar{\rho},u_2)$ are solutions to \eqref{st1.1} with the same $(F,G)$, and $\|(\sigma_1,u_1)\|_{{\mathcal{F}}^{4,5}}\leq\epsilon$, $\|(\sigma_2,u_2)\|_{{\mathcal{F}}^{4,5}}\leq\epsilon$, then $(\sigma_1,u_1)=(\sigma_2,u_2)$.
\end{theorem}
Let $(\sigma^*,u^*)$ be the stationary solution constructed in Theorem \ref{thm1.1}, and define the following function space:
\begin{eqnarray*}
&\mathfrak{C}(0,T;\mathcal{H}^{j,k})=&\{(\varrho,\omega)|\varrho(t,x)\in C^0(0,T;H^j)\cap C^1(0,T;H^{j-2}),\\
&&~~\omega(t,x)\in C^0(0,T;H^k)\cap C^1(0,T;H^{k-2})\}.
\end{eqnarray*}
Our second theorem is concerned with the nonlinear stability of the stationary solution.
\begin{theorem}\label{thm1.2}
There exist a positive constant $C$ and a small constant $\delta>0$ such that if $\|(\rho_0-\rho^*,u_0-u^*)\|_{4,3}\leq \delta$, the Cauchy problem \eqref{1.1} admits a unique global solution $(\rho,u)=(\varrho+\rho^*,\omega+u^*)$. Moreover, $(\varrho,\omega)\in \mathfrak{C}(0,\infty;\mathcal{H}^{4,3})$, $\ti \varrho \in L^2(0,\infty;{H}^{4})$, $\ti \omega \in L^2(0,\infty;{H}^{3})$, and the following estimate holds
\begin{eqnarray*}
\|( \varrho, \omega)(t)\|_{4,3}^2+\int\limits_0^t\|(\ti \varrho,\ti \omega)(s)\|_{4,3}^2ds\leq C\|( \varrho, \omega)(0)\|_{4,3}^2,
	\end{eqnarray*}
for any $t\geq 0$.
\end{theorem}
Finally, the decay rates of the solutions to the Cauchy problem \eqref{1.1} are stated as follows.
\begin{theorem}\label{thm1.3}
Under the assumptions of Theorem \ref{thm1.2}, and if $(\rho_0-\rho^*,\ti(\rho_0-\rho^*),m_0-m^*)\in \dot{H}^{-s}$ for some $s\in [0,\f32)$, we have
\begin{eqnarray}\label{1.3}
\|( \rho-\rho^*,\ti(\rho-\rho^*), m-m^*)(t)\|_{\dot{H}^{-s}}\leq C,
	\end{eqnarray}
for all $t\geq 0$. Moreover, the following decay rates hold
\begin{eqnarray}
  \|(\ti (\rho-\rho^*,m-m^*)(t)\|_{3,2}\leq C(1+t)^{-\f{1+s}{2}},
\end{eqnarray}
and
\begin{eqnarray}
\|({\rho-\rho^*}, m-m^*)(t)\|\leq C(1+t)^{-\f{s}{2}}.
\end{eqnarray}
\end{theorem}
Note that Lemma \ref{Ala3.3} implies that for $p\in (1,2]$, $L^p\subset \dot{H}^{-s}$ with $s=3(\f1p-\f12)\in [0,\f32)$, which implies the following decay result:
\begin{remark}
If we replace the $\dot{H}^{-s}$ assumptions in Theorem \ref{thm1.3} by $(\rho_0-\rho^*,\ti(\rho_0-\rho^*),m_0-m^*)\in L^{p}$ for $p\in (1,2]$, then we have
\begin{eqnarray*}
\|(\ti (\rho-\rho^*,m-m^*)(t)\|_{3,2}\leq C(1+t)^{-\f32(\f1p-\f12)-\f 12},
\end{eqnarray*}
and
\begin{eqnarray*}
\|({\rho-\rho^*}, m-m^*)(t)\|\leq C(1+t)^{-\f32(\f1p-\f12)}.
\end{eqnarray*}
\end{remark}
There is much literature concerned with the asymptotic behaviors of the solutions for the fluid dynamics. For the convergence rates for the compressible Navier-Stokes equations with with potential
forces, see, for example, the works in \cite{K1}-\cite{K4}. When the external force is given by the general form $F=\Div F_1+F_2$, the stationary solution is nontrivial in general, for the compressible Navier-Stokes equations, cf. \cite{[jj1]}-\cite{[jj3]}. For the compressible Navier-Stokes-Poisson equations, cf. \cite{HCZT}. For the compressible Navier-Stokes-Korteweg equations, cf. \cite{CZ}.

For the compressible quantum model, the time decay rates for higher-order spatial derivatives of solutions to the full compressible quantum model were established by Pu and Xu \cite{PX1} when the initial perturbation belongs to $(H^{N+2}\cap \dot{H}^{-s})\times (H^{N+1}\cap  \dot{H}^{-s})\times(H^N\cap \dot{H}^{-s})$ ($N\geq 3, s\in [0,\f{3}{2})$). Xie, Xi and Guo \cite{[XXG]} deduced faster time convergence rates for higher-order spatial derivatives of solutions when the initial data belongs to $L^1$, which improve the work of \cite{PG}. For the compressible quantum MHD system, the global existence and decay of classical solutions near the constant steady solution were established by Pu and Xu \cite{sPX1}. The time decay rates for higher-order spatial derivatives of solutions when the initial perturbation belongs to $\dot{H}^{-s}$ and $L^1$ were established in \cite{XXY} and \cite{xpg}, respectively. For the time decay rates for the compressible quantum Hall-magnetohydrodynamic model, one can refer to \cite{xpg2}. When the external force is considered, Xu and Pu \cite{XP} proved the optimal convergence rates of the magnetohydrodynamic model for quantum plasmas with potential force.


In this paper, we aim to establish the existence and stability of stationary solutions to the compressible quantum model with the general force $F=\Div F_1+F_2$ and the time decay rates of the solutions when the initial perturbation belongs to $\dot{H}^{-s}$ with $0 \leq s < \f32$. Compared with the classical hydrodynamic equations, the appearance of the higher-order quantum term will lead to new difficulties. First, in the same spirit as \cite{[jj1],[jj2]}, the weighted $L^2$ theory for the linearized equation of \eqref{aA1.1} can be established by elaborate analysis. Then, together with the weighted $L^\infty$ estimates on
solutions to the linearized problem \eqref{aA1.2}, Theorem \ref{thm1.1} can be proved by the contraction mapping principle. As for the nonlinear stability of the stationary solution, based on
the properties obtained on the stationary solution and some delicate estimates, we can deduce the a prior estimates.
Another purpose of this paper is to derive the negative Sobolev norm estimates and the decay rates
of the solution towards the stationary solution. Since the stationary
solution is nontrivial, we consider equation \eqref{4.1} to obtain the decay rates of the solution by using the $L^p$-$L^q$
estimates.

The rest of the paper is organized as follows. In Section 2, some useful lemmas are given, which will be used throughout this paper. In Section 3, we establish the existence and uniqueness of stationary solutions of the problem \eqref{st1.1}. In Section 4, the stability of the stationary solution is obtained. In the last section, we aim to prove the decay rates for the nonlinear system.
\section{Preliminaries}
\begin{lemma}[Gagliardo-Nirenberg inequality]\label{fflma2.1}
Let $0\leq m,\alpha\leq l$ and the function $f\in C_0^\infty(\R^3)$, then we have
\begin{eqnarray*}
			&&\|\ti^\alpha f\|_{L^p}\leq\|\ti^m f\|_{L^2}^{1-\theta}\|\ti^l f\|_{L^2}^\theta,
		\end{eqnarray*}
where $0\leq \theta \leq 1$ and $\alpha$ satisfies
\begin{eqnarray*}
			\f 1p-\f \alpha 3=\bigg(\f 12-\f m3\bigg)(1-\theta)+\bigg(\f 12-\f l3\bigg)\theta.
		\end{eqnarray*}
\end{lemma}
\begin{lemma}[see \cite{PG1}]\label{flma2.3}
Let $m\geq 1$ be an integer and define the commutator
\begin{eqnarray*}
[\ti^m,f]g=\ti^m(fg)-f\ti^m g.
\end{eqnarray*}
Then there exists some constant $C>0$ such that
\begin{eqnarray*}
&&\|\ti^m(fg)\|_{L^p}\leq C(\|f\|_{L^{p_1}}\|g\|_{\dot{H}^{m,p_2}}+\|g\|_{L^{p_3}}\|f\|_{\dot{H}^{m,p_4}}),\\
&&\|[\ti^m,f]g\|_{L^p}\leq C(\|\ti f\|_{L^{p_1}}\|g\|_{\dot{H}^{{m-1},p_2}}+\|g\|_{L^{p_3}}\|f\|_{\dot{H}^{m,p_4}}),
\end{eqnarray*}
where $f,g\in \mathcal{S}$, and $p_2,p_4\in (1,\infty)$ such that
\begin{eqnarray*}
\f1p=\f1{p_1}+\f1{p_2}=\f1{p_3}+\f1{p_4}.
\end{eqnarray*}
\end{lemma}
Next, we will give some useful definitions related to the negative Sobolev norms.
\begin{definition}
The operator $\Lambda^s$ is defined by
\begin{eqnarray*}
			\Lambda^{-s}g(x)=\int_{\R^3}|\xi|^s\hat{g}(\xi)e^{2\pi ix\cdot \xi}d\xi,
		\end{eqnarray*}
where $s\in \R$ and $\hat{g}$ is the Fourier transform of $g$.
\end{definition}
\begin{definition}
    $\dot{H}^s$ is defined as the homogeneous Sobolev space of $g$, with the norm
\begin{eqnarray*}
    \|g\|_{\dot{H}^s}:=\|\Lambda^sg\|_{L^2}=\||\xi|^sg\|_{L^2}.		
\end{eqnarray*}
\end{definition}
The index $s$ can be non-positive real numbers. In the sequel, we will use the index $-s$ with $s\geq 0$ for convenience.

The following lemmas are useful during the proof for decay rates of the solutions, see \cite{W5}, \cite{SEN} and \cite{K4}.
\begin{lemma}\label{Ala3.2}
Let $s\geq 0$ and $l\geq0$, then we have
\begin{eqnarray*}
    \|\ti^l g\|_{L^2}\leq\|\ti^{l+1}g\|_{L^2}^{1-\theta}\|g\|_{\dot{H}^{-s}}^\theta, ~\mathrm{where}~ \theta=\f{1}{l+s+1}.	
\end{eqnarray*}
\end{lemma}
\begin{lemma}\label{Ala3.3}
Let $0< s< 3, 1< p< q < \infty, 1/q+s/3=1/p$, then we have
\begin{eqnarray*}
    \|\Lambda^{-s}g\|_{L^q}\leq\|g\|_{L^p}.
\end{eqnarray*}
\end{lemma}
\begin{lemma}\label{la3.2}
Let $r_1>1$ and $r_2\in[0,r_1]$, then there exists a constant $C_1=\f{2^{r_2+1}}{r_1-1}$ such that
\begin{eqnarray*}
  \int_0^t(1+t-\tau)^{-r_1}(1+\tau)^{-r_2}d\tau\leq C_1(1+t)^{-r_2}.
\end{eqnarray*}
\end{lemma}
\section{Existence of stationary solutions}
For the stationary problem \eqref{st1.1}, we consider $\sigma=\rho-\bar{\rho}$, and set $\bar{\rho}=\gamma=1$ for simplicity,
\begin{equation}\label{aA1.1}
		\left\{ \begin{aligned}
			&\mathrm{div}u+\f{u}{\rho}\ti\sigma=\f{G}{\rho},\\
 			&\rho u\cdot\ti u-(\mu\Delta u+(\mu+\lambda)\ti\Div u)+\ti\sigma-\f{\hbar^2}{4}\ti\Delta\sigma\\
 &=\rho F-uG-(P'(\rho)-1)\ti\sigma-\f{\hbar^2}{4}\bigg(\f{|\ti \sigma|^2\ti \sigma}{\rho^2}-\f{\ti \sigma \triangle \sigma}{\rho}-\f{\ti \sigma\cdot \ti^2\sigma}{\rho}\bigg).
		\end{aligned} \right.
	\end{equation}
In order to construct a solution to the stationary problem \eqref{st1.1} by the contraction mapping principle in $\dot{\mathcal{F}}^{4,5}_{\epsilon}$, we consider the following iteration system:
\begin{equation}\label{aA1.2}
		\left\{ \begin{aligned}
			&\mathrm{div}u+\f{\tilde{u}}{1+\tilde{\sigma}}\ti\sigma=g,\\
 			&-(\mu\Delta u+(\mu+\lambda)\ti\Div u)+\ti\sigma-\f{\hbar^2}{4}\ti\Delta\sigma=- \tilde{u}\cdot\ti \tilde{u}+\tilde{f},
		\end{aligned} \right.
	\end{equation}
where $(\tilde{\sigma},\tilde{u})\in\dot{\mathcal{F}}^{4,5}_{\epsilon}$ is given, and
\begin{equation*}
		\left\{ \begin{aligned}
			&g=\f{G}{1+\tilde{\sigma}},\\
 			&\tilde{f}=-\tilde{\sigma} \tilde{u}\cdot\ti \tilde{u}+(1+\tilde{\sigma}) F-\tilde{u}G-(P'(1+\tilde{\sigma})-1)\ti\tilde{\sigma}
 -\f{\hbar^2}{4}\bigg(\f{|\ti \tilde{\sigma}|^2\ti \tilde{\sigma}}{(1+\tilde{\sigma})^2}-\f{\ti \tilde{\sigma} \triangle \tilde{\sigma}}{1+\tilde{\sigma}}-\f{\ti \tilde{\sigma}\cdot \ti^2\tilde{\sigma}}{1+\tilde{\sigma}}\bigg).
		\end{aligned} \right.
	\end{equation*}
The proof of the existence and uniqueness of stationary solutions to the problem \eqref{st1.1} is similar to that in \cite{CZ}, and we omit the details here. From the weighted $L^2$ theory for the linearized problem of \eqref{aA1.1} and some analysis on $g$ and $\tilde{f}$, the following lemma can be obtained:
\begin{lemma}\label{la3.4}
Let $(G,F)\in\mathcal{H}^{4,3}$ satisfy $K_0<\infty$. There exists a constant $\epsilon_0>0$ such that if $\epsilon\leq \epsilon_0$, \eqref{aA1.2} with $(\tilde{\sigma},\tilde{u})\in \dot{{\mathcal{F}}}^{4,5}_\epsilon$ admits a solution $(\sigma,u)\in \hat{\mathcal{H}}^{6,5}$ satisfying
\begin{align}\label{a107}
\|(\sigma,u)\|_{L^6}+\sum\limits_{\nu=1}^4\|(1+|x|)^\nu (\nabla^\nu\sigma,\nabla^{\nu+1} \sigma,\nabla^{\nu+2}\sigma)\|+\sum\limits_{\nu=1}^5\|(1+|x|)^{\nu-1}\nabla^\nu u\|\leq C(\epsilon^2+K_0),
\end{align}
where the positive constant $C$ depends only on $\mu,\lambda,\hbar$.
\end{lemma}
Due to Lemma \ref{la3.4}, define the map $T:\dot{{\mathcal{F}}}^{4,5}_\epsilon\rightarrow\hat{\mathcal{H}}^{6,5}$ by $(\sigma,u)=T(\tilde{\sigma},\tilde{u})$.
By the Helmholtz decomposition and the Fourier transform, the solution of \eqref{aA1.2} can be formulated as follows:
\begin{align}\label{a108}
v=\omega+\nabla p,~~\sigma-{\f{\hbar^2}{4}\Delta\sigma}=\Psi+(2\mu+\lambda)\Delta p,
\end{align}
where
\begin{equation}{}
		\left \{ \begin{aligned}
&p(x)=E_0\ast R(x),\\
&\omega_j(x)=\sum_{i=1}^{3}E_{ij}\ast f_i(x),\\
&\Psi=E_0\ast\mathrm{div}f,
		\end{aligned} \right.
	\end{equation}
and
\begin{equation}{}
		\left \{ \begin{aligned}
&E_{ij}(x)=\f{1}{8\pi\mu}\big(\f{\delta_{ij}}{|x|}-\f{x_ix_j}{|x|^3}\big),\\
&E_0=-\f{1}{4\pi|x|},\\
&f_i(x)=-\tilde{u}\cdot\nabla \tilde{u}_i+\tilde{f}_i(x),\\
&R(x)=-\f{\tilde{u}}{1+\tilde{\sigma}}\ti\sigma+\f{G}{1+\tilde{\sigma}}.
		\end{aligned} \right.
	\end{equation}
Then, the $L^\infty$ estimates for the solution to \eqref{aA1.2} can be deduced as follows:
\begin{lemma}
Let $(G,F)\in\mathcal{H}^{4,3}$ satisfy $K<\infty$.
If $(\sigma,u)\in \hat{\mathcal{H}}^{6,5}$ which satisfies \eqref{a107} is a solution to \eqref{aA1.2} with $(\tilde{\sigma},\tilde{u})\in \dot{{\mathcal{F}}}^{4,5}_\epsilon$, then
\begin{align}\label{}
\sum\limits_{\nu=0}^1\|(1+|x|)^2\nabla^\nu \sigma\|_{L^\infty}+\sum\limits_{\nu=0}^1\|(1+|x|)^{\nu+1}\nabla^\nu u\|_{L^\infty}+\|(1+|x|)^2\nabla^2 u\|_{L^\infty}\leq C(\epsilon^2+K),
\end{align}
where the positive constant $C$ depends only on $\mu,\lambda,\hbar$.
\end{lemma}
\begin{proof}
For the estimate on $f$. Since $(\tilde{\sigma},\tilde{u})\in \dot{{\mathcal{F}}}^{4,5}_\epsilon$, there exist $\tilde{V}_1$ and $\tilde{V}_2$ such that $\Div \tilde{v}=\Div \tilde{V}_1+\tilde{V}_2$, and
\begin{eqnarray}\label{a113}
\|(1+|x|)^3\tilde{V}_1\|_{L^\infty}+\|(1+|x|)^{-1}\tilde{V}_2\|_{L^1}\leq \epsilon.
\end{eqnarray}
Then
\begin{eqnarray*}
{f}_i&=&-\tilde{\rho} \tilde{u}\cdot\ti \tilde{u}_i+\tilde{\rho} F_i-\tilde{u}_iG-[(P'(\tilde{\rho} )-1)\ti\tilde{\sigma}]_i
 -\f{\hbar^2}{4}\bigg(\f{|\ti \tilde{\sigma}|^2\ti \tilde{\sigma}}{\tilde{\rho} ^2}-\f{\ti \tilde{\sigma} \triangle \tilde{\sigma}}{\tilde{\rho} }-\f{\ti \tilde{\sigma}\cdot \ti^2\tilde{\sigma}}{\tilde{\rho} }\bigg)_i\\
&=&\Div(-\tilde{\rho}\tilde{u}\tilde{u}_i+\tilde{\rho}\tilde{u}_i\tilde{V}_1+\tilde{\rho}F_{1,i})
+\{-\tilde{\rho}\tilde{V}_1\cdot\ti\tilde{u}_i-\tilde{u}_i\tilde{V}_1\cdot\ti\tilde{\rho}+
\tilde{\rho}\tilde{u}_i\tilde{V}_2+\tilde{\rho}F_{2,i}+\tilde{u}_i\tilde{u}\ti\tilde{\rho}\\
&&-[(P'(\tilde{\rho} )-1)\ti\tilde{\sigma}]_i
 -\f{\hbar^2}{4}\bigg(\f{|\ti \tilde{\sigma}|^2\ti \tilde{\sigma}}{\tilde{\rho} ^2}-\f{\ti \tilde{\sigma} \triangle \tilde{\sigma}}{\tilde{\rho}}-\f{\ti \tilde{\sigma}\cdot \ti^2\tilde{\sigma}}{\tilde{\rho} }\bigg)_i-\ti\tilde{\rho}F_{1,i}-\tilde{u}_iG\}\\
&=&\Div f_{1,i}+f_{2,i}.
\end{eqnarray*}
By virtue of \eqref{a113}, the Sobolev inequality and mean value theorem, we have
\begin{eqnarray*}
\|(1+|x|)^3f_i\|_{L^\infty}+\|(1+|x|)^{2}f_{1,i}\|_{L^\infty}+\|f_{2,i}\|_{L^1}\leq C(\epsilon^2+K_1),
\end{eqnarray*}
and
\begin{eqnarray*}
\|(1+|x|)^3\ti f_i\|_{L^\infty}+\|(1+|x|)^{2}f_{i}\|_{L^\infty}\leq C(\epsilon^2+K_1),
\end{eqnarray*}
where
\begin{eqnarray*}
K_1&=&\|(1+|x|)^3(F,G,\ti F,\ti G)\|_{L^\infty}+\|(1+|x|)^2F_1\|_{L^\infty}+\|F_2\|_{L^1}.
\end{eqnarray*}
Hence, by (i) and (ii) of Lemma 2.5 in {\cite{CZ}}, we get
\begin{eqnarray}\label{a114}
|x||\omega_j|,|x|^2(|\ti \omega_j|+|\ti^2\omega_j|+|\Psi|+|\ti\Psi|)\leq C(\epsilon^2+K_1).
\end{eqnarray}
For the estimate on $p$, See {\cite{[jj2]}},
\begin{eqnarray}\label{a115}
|x||\ti p|,|x|^2(|\ti^2 p|+|\ti^3 p|)\leq C(\epsilon^2+K_0+K_2),
\end{eqnarray}
where
\begin{eqnarray*}
K_2&=&\|(1+|x|)^2G\|_{L^\infty}+\|(1+|x|)^3(\ti G,\ti^2 G)\|_{L^\infty}.
\end{eqnarray*}
Then, we have from \eqref{a108}, \eqref{a114} and \eqref{a115} that
\begin{eqnarray}\label{a116}
|x||u|,|x|^2(|\ti u|+|\ti^2 u|)\leq C(\epsilon^2+K_0+K_1+K_2).
\end{eqnarray}
Next, for the estimate on $\sigma$, by the Fourier transform, we have from \eqref{a108} and {\cite{LCE}}
\begin{eqnarray*}
\sigma(x)&=&\f{1}{(2\pi)^{3/2}}\mathcal{F}^{-1}\bigg(\f{1}{1+\hbar^2|\xi|^2/4}\bigg)\ast L(\Psi,p)\\
&=&\f{1}{\pi^{3/2}\hbar^3}\int_{\mathbb{R}^3}\bigg(\int_0^\infty e^{-t-\f{|y|^2}{\hbar^2 t}}t^{-\f32}dt\bigg)
(\Psi+(2\mu+\lambda)\Delta p)(x-y)dy.
\end{eqnarray*}
Hence, it follows from {\cite{CZ}} that
\begin{eqnarray*}
|x|^2(|\sigma|+|\ti \sigma|)\leq C(\epsilon^2+K).
\end{eqnarray*}
The case of $|x|<1$ can be obtained from the Sobolev inequality and \eqref{a107}. This completes the proof of this lemma.
\end{proof}
The following proposition shows that the solution $(\sigma,u)\in \dot{{\mathcal{F}}}^{4,5}_\epsilon$ and the solution map $T$ is contractive.
\begin{proposition}\label{pro2.3}
There exist a constant $c_0>0$ and a sufficiently small constant $\epsilon>0$ such that if $(G,F)\in\mathcal{H}^{4,3}$ satisfies
\begin{eqnarray*}
    K+\|(1+|x|)^{-1}G\|_{L^1}\leq c_0\epsilon,
\end{eqnarray*}
then \eqref{aA1.2} with $(\tilde{\sigma},\tilde{u})\in \dot{{\mathcal{F}}}^{4,5}_\epsilon$ admits a solution $(\sigma,u)=T(\tilde{\sigma},\tilde{u})\in \dot{{\mathcal{F}}}^{4,5}_\epsilon$. Moreover, for
$(\sigma^j,u^j)\in \dot{{\mathcal{F}}}^{4,5}_\epsilon$ with $(j=1,2)$, we have
\begin{eqnarray*}
&&\|(\sigma^1-\sigma^2,u^1-u^2)\|_{{{\mathcal{F}}}^{4,5}}+\|(1+|x|)^3({V}_1^1-{V}^2_1)\|_{L^\infty}+\|(1+|x|)^{-1}({V}^1_2-{V}_2^2)\|_{L^1}\\
&&\leq\f12\{\|(\tilde{\sigma}^1-\tilde{\sigma}^2,\tilde{u}^1-\tilde{u}^2)\|_{{{\mathcal{F}}}^{4,5}}+\|(1+|x|)^3(\tilde{{V}}_1^1-\tilde{{V}}^2_1)\|_{L^\infty}+\|(1+|x|)^{-1}(\tilde{{V}}^1_2-\tilde{{V}}_2^2)\|_{L^1}\},
	\end{eqnarray*}
where $(\tilde{V}_1^j,\tilde{V}_2^j)$ are defined by
\begin{eqnarray*}
&&\Div \tilde{u}^j=\Div \tilde{V}_1^j+\tilde{V}_2^j,\|(1+|x|)^3\tilde{V}_1^j\|_{L^\infty}+\|(1+|x|)^{-1}\tilde{V}_2^j\|_{L^1}\leq\epsilon,
	\end{eqnarray*}
and $({V}_1^j,{V}_2^j)$ are defined by
\begin{eqnarray*}
&&{V}_1^j=-\f{\tilde{u}^j}{\tilde{\rho}^j}\sigma^j,~~~~~{V}_2^j=\Div\bigg(\f{\tilde{u}^j}{\tilde{\rho}^j}\bigg)\sigma^j+\f{G}{\tilde{\rho}^j}.
	\end{eqnarray*}
\end{proposition}
By Proposition \ref{pro2.3}, the contraction mapping principle implies the existence and uniqueness of the solution to the problem\eqref{st1.1}, and thus the proof of Theorem \ref{thm1.1} is complete.
\section{Stability of the stationary solution}
The corresponding stationary solution to the problem \eqref{aA1.1} is denoted by $(\rho^*=\sigma^*+\bar{\rho},u^*)$.
We set $\varrho(t)=\rho(t)-\rho^*$, $\omega(t)=u(t)-u^*$ and $\bar{\rho}=1$, the problem \eqref{1.1} can be rewritten as
\begin{equation}\label{A3.1}
		\left\{ \begin{aligned}
			&\varrho_t(t)+\mathrm{div} [(\varrho(t)+\rho^*)\omega(t)]=-\mathrm{div}(u^*\varrho(t)),\\
 			&\omega_t(t)-\f{1}{\rho^*}(\mu\Delta \omega(t)+(\mu+\lambda)\ti\Div \omega(t))+A(t)\ti\varrho(t)-\f{\hbar^2}{4}\ti\Delta\varrho(t)=f(t),\\
            &(\varrho,\omega)(t)|_{t=0}=(\rho_0-\rho^*,u_0-u^*).
		\end{aligned} \right.
	\end{equation}
	Here $A(t)$ and $f(t)$ are defined as
	\begin{eqnarray*}
A(t)&=&\f{P'(\varrho(t)+\rho^*)}{\varrho(t)+\rho^*},\\
f(t)&=&-(A(t)-A^*)\ti\rho^*-\omega(t)\cdot\ti(\omega(t)+u^*)-u^*\cdot\ti \omega(t)-\bigg(\f{\omega(t)+u^*}{\varrho(t)+\rho^*}-\f{u^*}{\rho^*}\bigg)G\\
        &&-\f{\varrho(t)}{\rho^*(\varrho(t)+\rho^*)}(\mu\Delta(\omega(t)+u^*)-(\mu+\lambda)\ti\Div(\omega(t)+u^*))\\
        &&+\f{\hbar^2}4\bigg(\bigg(\f{\rho^*-1}{\rho^*}-\f{\varrho(t)+\rho^*-1}{\varrho(t)+\rho^*}\bigg)\ti\Delta {\rho^*}-\f{\varrho(t)+\rho^*-1}{\varrho(t)+\rho^*}\ti\Delta {\varrho(t)}\bigg)\\
		&&+\f{\hbar^2}4\bigg(\bigg(\f{|\ti (\varrho(t)+\rho^*)|^2}{(\varrho(t)+\rho^*)^3}-\f{|\ti {\rho^*}|^2}{(\rho^*)^3}\bigg)\ti {\rho^*}+\f{|\ti (\varrho(t)+\rho^*)|^2}{(\varrho(t)+\rho^*)^3}\ti \varrho(t)\\
		&&~~~~~~~~~-\bigg(\f{\ti (\varrho(t)+\rho^*) }{(\varrho(t)+\rho^*)^2}-\f{\ti {\rho^*} }{(\rho^*)^2}\bigg)\Delta {\rho^*}-\f{\ti (\varrho(t)+\rho^*) }{(\varrho(t)+\rho^*)^2}\Delta \varrho(t)\\
		&&~~~~~~~~~-\bigg(\f{\ti(\varrho(t)+\rho^*)}{(\varrho(t)+\rho^*)^2}-\f{\ti {\rho^*} }{(\rho^*)^2}\bigg)\ti^2{\rho^*}-\f{\ti (\varrho(t)+\rho^*) }{(\varrho(t)+\rho^*)^2}\ti^2\varrho(t)\bigg),
	\end{eqnarray*}
with $A(t)=A(\varrho(t)+\rho^*)$ and $A^*=A(\rho^*)$.

In this section, the proof of Theorem \ref{thm1.2} consists of two steps. First, the local existence result is the following:
\begin{proposition}\label{Pro3.1}
There exists a constant $t_0>0$ such that the initial value problem \eqref{A3.1} with the initial data $(\varrho,\omega)(0)\in\mathcal{H}^{4,3}$ admits a unique solution $(\varrho,\omega)(t)\in \mathfrak{C}(0,t_0;\mathcal{H}^{4,3})$. Moreover, we have
\begin{eqnarray*}
\|(\varrho,\omega)(t)\|_{4,3}^2\leq C\|(\varrho,\omega)(0)\|_{4,3}^2,
\end{eqnarray*}
for any $t\in [0,t_0]$, where $C$ is a positive constant.
\end{proposition}
Next, the other result is the a priori estimate.
\begin{proposition}\label{Pro3.2}
Let $(\varrho,\omega)(t)\in \mathfrak{C}(0,t_1;\mathcal{H}^{4,3})$ be a solution to the initial value problem \eqref{A3.1} for some positive constant $t_1$. Then there exists a constant $\delta_0>0$ such that if $0<\delta\leq \delta_0$ and $\sup\limits_{0\leq t\leq t_1}\|(\varrho,\omega)(t)\|_{4,3}+\|(\sigma^*,u^*)\|_{\mathcal{F}^{4,5}}\leq\delta$, we have
\begin{eqnarray*}
\|( \varrho, \omega)(t)\|_{4,3}^2+\int\limits_0^t\|(\ti \varrho,\ti \omega)(s)\|_{4,3}^2ds\leq C\|( \varrho, \omega)(0)\|_{4,3}^2,
	\end{eqnarray*}
for any $t\in [0,t_0]$, where the constant $C>0$ depends only on $\mu,\lambda,\hbar$.
\end{proposition}
The proof of Proposition \ref{Pro3.1} is similar to that in \cite{HHDL} by using the iteration argument. We omit it here. The aim of this {section} is to derive the a priori energy estimates in Proposition \ref{Pro3.2} for the system \eqref{A3.1}.
We make a priori assumption that $\|(\varrho,\omega)(t)\|_{4,3}$ is small enough, with the fact that $(\sigma^*,u^*)\in {\mathcal{F}}^{4,5}_{\epsilon}$, we assume that there exists a small constant $\delta>0$ such that
\begin{eqnarray*}
\|(\varrho,\omega)(t)\|_{4,3}+\|(\sigma^*,u^*)\|_{\mathcal{F}^{4,5}}\leq\delta.
	\end{eqnarray*}
Hence, by using the Sobolev inequality, we have
\begin{eqnarray*}
\f{\bar{\rho}}{2}\leq \varrho+\rho^*\leq \f{3\bar{\rho}}{2},~~\f{3\bar{\rho}}{4}\leq \rho^*\leq \f{5\bar{\rho}}{4}.
	\end{eqnarray*}
The following lemma is concerned with some estimates on $f(t)$ and its derivatives.
\begin{lemma}
For a multi-index $\alpha$ with $0\leq |\alpha|\leq 3$, if we write $\partial_x^\alpha f(t)$ as
\begin{eqnarray*}
&&\ti^k f(t)=-\f{\varrho(t)}{\rho^*(\varrho(t)+\rho^*)}\ti^k(\mu\Delta\omega(t)+(\mu+\lambda)\ti\Div \omega(t))+F_k(t)+\f{\hbar^2}4\ti^k{J(t)},
\end{eqnarray*}
then $F_\alpha(t)$ satisfies
\begin{equation}{F_k(t)\leq C}
		\left \{ \begin{aligned}
&|\ti u^*||\omega(t)|+(|u^*|+|\omega(t)|)|\ti \omega(t)|\\
&~~+(|\ti \rho^*|+|\ti^2u^*|)|\varrho(t)|+(|\varrho(t)|+|\omega(t)|)|G|,~~if~ k=0,\\
&|\ti^{k+1}u^*||\omega(t)|+\sum\limits_{\nu=1}^{k+1}|\ti^\nu\omega(t)|+\sum\limits_{\nu=1}^{k+1}(|\ti^\nu\rho^*|+|\ti^{\nu+1}u^*)||\varrho(t)|\\
 &~~+\sum\limits_{\nu=1}^{k}|\ti^\nu\varrho(t)|+(|\varrho(t)|+|\omega(t)|)\sum\limits_{\nu=0}^{k}|\ti^\nu G|+R_F^{k}(t),~~if~ k=1,2,3,
		\end{aligned} \right.
	\end{equation}
where
\begin{align*}
&R_F^{k}(t)=0,~~k=1,2,\\
&\|R_F^{3}(t)\|_{L^{\f32}}\leq C\delta\|(\ti^2\varrho,\ti^3\omega)(t)\|_{1,0}.
\end{align*}
Moreover, by virtue of the mean value theorem, $J(t)$ satisfies
\begin{align}\label{J}
&J(t)\sim \varrho(t)\nabla\Delta\rho^*+(\varrho(t)+\rho^*-1)\nabla\Delta\varrho(t)+|\nabla\rho^*|^2\varrho(t)+|\nabla\varrho(t)|^2\nabla\rho^*\notag\\
&+|\nabla\rho^*|^2\nabla\varrho(t)+|\nabla\varrho(t)|^2\nabla\varrho(t)+\varrho(t)\nabla\rho^*\Delta\rho^*+\nabla\varrho(t)\Delta\rho^*+\nabla\varrho(t)\Delta\varrho(t)\notag\\
&+\nabla\rho^*\Delta\varrho(t)+\varrho(t)\nabla\rho^*\ti^2\rho^*+\nabla\varrho(t)\ti^2\rho^*+\nabla\varrho(t)\ti^2\varrho(t)+\nabla\rho^*\ti^2\varrho(t).
\end{align}
\end{lemma}
In the following lemma, we will give some estimates on $\ti \omega$ and its derivatives up to $\ti^4 \omega$.
\begin{lemma}\label{As3.2}
Let $(\varrho,\omega)\in \mathfrak{C}(0,t_1;\mathcal{H}^{4,3})$ be a solution to the initial value problem \eqref{A3.1}. Then there exist constants $\delta_0,\varepsilon_0>0$ and $d_1>0$ such that if $0<\delta\leq \delta_0$, $0<\varepsilon\leq \varepsilon_0$, and $\|(\varrho,\omega)(t)\|_{4,3}+\|(\sigma^*,u^*)\|_{\mathcal{F}^{4,5}}\leq\delta$, we have
 \begin{eqnarray}\label{As3.10}
\f{d}{dt}(\|\varrho\|^2+\langle\hat{A}(t)\ti \varrho,\ti \varrho\rangle+\langle\tilde{A}(t)\omega,\omega\rangle)+d_1\|\ti\omega\|^2\leq C\delta\|\ti \varrho\|^2_1.
	\end{eqnarray}
Moreover, for $1\leq k \leq 3$, it holds that
\begin{eqnarray}\label{As3.11}
&&\f{d}{dt}(\|\ti^k\varrho\|^2+\langle\hat{A}(t)\ti^{k+1} \varrho,\ti^{k+1} \varrho\rangle+\langle\tilde{A}(t)\ti^{k}\omega,\ti^{k}\omega\rangle)+d_1\|\ti^{k+1}\omega\|^2\notag\\
&&~~\leq C(\delta+\varepsilon)\|\ti (\varrho,\omega)\|^2_{k+1,k-1}+C\varepsilon^{-1}\|\ti^k \omega\|^2,
	\end{eqnarray}
where
\begin{eqnarray*}
\hat{A}(t)=\f{\varrho+\rho^*}{P'(\varrho+\rho^*)},~~\tilde{A}(t)=\f{(\varrho+\rho^*)^2}{P'(\varrho+\rho^*)},
	\end{eqnarray*}
 and the constant $C>0$ depends only on $\mu,\lambda,\hbar$.
\end{lemma}
\begin{proof}
Taking $k$-$th$ spatial derivatives to $\eqref{A3.1}_1,~\eqref{A3.1}_2$, and multiplying the resultant equation by $\ti^k \varrho,~\tilde{A}\ti^k \omega$, respectively, summing up them yields
\begin{eqnarray*}
    && \f{1}{2}\f{d}{dt}\|\ti^k\varrho\|^2+\langle\tilde{A}\ti^k \omega_t,\ti^k \omega\rangle-\bigg\langle\f{\tilde{A}}{\rho^*}\ti^k(\mu\Delta\omega+(\mu+\lambda)\ti\Div \omega),\ti^k \omega\bigg\rangle\\
    &&=\f{\hbar^2}{4}\langle\ti^k\ti\Delta\varrho,\tilde{A}\ti^k \omega\rangle+\langle\ti^k(u^*\varrho)+\sum_{l<k}C_k^l\ti^{k-l}(\rho^*+\varrho)\ti^l\omega,\ti^{k+1}\varrho\rangle\\
    &&~~+\bigg\langle\ti^k f+\sum_{l<k}C_k^l\bigg\{\bigg(\ti^{k-l}\f{1}{\rho^*}\bigg)\ti^l(\mu\Delta\omega+(\mu+\lambda)\ti\Div \omega)-(\ti^{k-l}A)\ti^{l+1}\varrho\bigg\},\tilde{A}\ti^k \omega\bigg\rangle.
\end{eqnarray*}
By integrating by parts, we have
\begin{eqnarray*}
    &&\f{1}{2} \f{d}{dt}(\|\ti^k\varrho\|^2+\langle\tilde{A}\ti^k \omega,\ti^k \omega\rangle)+\bigg\langle\f{\tilde{A}}{\rho^*}\ti^{k+1}\omega,\ti^{k+1}\omega\bigg\rangle\\
    &&\leq \f{\hbar^2}{4}\langle\ti^k\ti\Delta\varrho,\tilde{A}\ti^k \omega\rangle+|\langle\ti^k(u^*\varrho),\ti^{k+1}\varrho\rangle|\\
    &&~~+\mu\bigg|\bigg\langle\ti\bigg(\f{\tilde{A}}{\rho^*}\bigg)\ti^k\omega,\ti^{k+1}\omega\bigg\rangle\bigg|+(\mu+\lambda)\bigg|\bigg\langle\ti\bigg(\f{\tilde{A}}{\rho^*}\bigg)\cdot \ti^k\omega,\ti^k\Div \omega\bigg\rangle\bigg|\\
    &&~~+\bigg|\bigg\langle\sum_{l<k}C_k^l\bigg\{\bigg(\ti^{k-l}\f{1}{\rho^*}\bigg)\ti^l(\mu\Delta\omega+(\mu+\lambda)\ti\Div \omega)-(\ti^{k-l}A)\ti^{l+1}\varrho\bigg\},\tilde{A}\ti^k \omega\bigg\rangle\bigg|\\
    &&~~+|\langle\sum_{l<k}C_k^l\ti^{k-l}(\rho^*+\varrho)\ti^l\omega,\ti^{k+1}\varrho\rangle|+|\langle\ti^k f,\tilde{A}\ti^k \omega\rangle|+\f{1}{2}|\langle\tilde{A}_t\ti^k \omega,\ti^k \omega\rangle|\\
    &&=\sum\limits_{j=1}^8I_j.
\end{eqnarray*}
For the term $I_1$, by using integrations by parts and $\eqref{A3.1}_1$, we deduce
\begin{align*}
I_1&=-\f{\hbar^2}{4}\langle\ti^k\Delta\varrho,\Div(\tilde{A}\ti^k \omega)\rangle\\
&=-\f{\hbar^2}{4}\langle\ti^k\Delta\varrho,(\ti\hat{A})(\rho^*+\varrho)\ti^k \omega+\hat{A}\Div\{(\rho^*+\varrho)\ti^k \omega\}\rangle\\
&=-\f{\hbar^2}{4}\langle\ti^k\Delta\varrho,(\ti\hat{A})(\rho^*+\varrho)\ti^k \omega\rangle+\f{\hbar^2}{4}\langle\ti^k\Delta\varrho,\hat{A}\ti^k\varrho_t\rangle\\
&~~~~+\f{\hbar^2}{4}\sum_{l<k}C_k^l\langle\ti^k\Delta\varrho,\hat{A}\Div\{\ti^{k-l}(\rho^*+\varrho)\ti^l \omega\}\rangle+\f{\hbar^2}{4}\langle\ti^k\Delta\varrho,\hat{A}\ti^{k+1}(\varrho u^*)\rangle
\end{align*}
\begin{align*}
&=-\f{\hbar^2}{4}\langle\ti^k\Delta\varrho,(\ti\hat{A})(\rho^*+\varrho)\ti^k \omega\rangle-\f{\hbar^2}{8}\f{d}{dt}\langle\ti^{k+1}\varrho,\hat{A}\ti^{k+1}\varrho\rangle\\
&~~~~+\f{\hbar^2}{8}\langle\ti^{k+1}\varrho,\hat{A}_t\ti^{k+1}\varrho\rangle+\f{\hbar^2}{4}\langle\ti^{k+1}\varrho,\ti\hat{A}\Div\{(\rho^*+\varrho)\ti^k \omega\}\rangle\\
&~~~~+\f{\hbar^2}{4}\sum_{l<k}C_k^l\langle\ti^{k+1}\varrho,\ti\hat{A}\Div\{\ti^{k-l}(\rho^*+\varrho)\ti^l \omega\}\rangle+\f{\hbar^2}{4}\langle\ti^{k+1}\varrho,\ti\hat{A}\ti^{k+1}(\varrho u^*)\rangle\\
&~~~~+\f{\hbar^2}{4}\sum_{l<k}C_k^l\langle\ti^k\Delta\varrho,\hat{A}\Div\{\ti^{k-l}(\rho^*+\varrho)\ti^l \omega\}\rangle+\f{\hbar^2}{4}\langle\ti^k\Delta\varrho,\hat{A}\ti^{k+1}(\varrho u^*)\rangle\\
&=I_{11}+I_{12}+I_{13}+I_{14}+I_{15}+I_{16}+I_{17}+I_{18}.
\end{align*}
For $I_{11}$, we have
\begin{align*}
I_{11}\leq&\|\ti^{k+2}\varrho\|\|\ti(\varrho,\rho^*)\|_{L^3}\|\ti^k\omega\|_{L^6}\\
\leq& C\delta\|\ti(\varrho,\omega)\|_{k+1,k}^2.
\end{align*}
For $I_{13}$, we have
\begin{align*}
I_{13}=& \f{\hbar^2}{8}\langle \hat{B}\varrho_t\ti^{k+1}\varrho,\ti^{k+1}\varrho\rangle\\
=& -\f{\hbar^2}{8}\langle \hat{B}\Div(u^*\varrho+(\rho^*+\varrho)\omega)\ti^{k+1}\varrho,\ti^{k+1}\varrho\rangle\\
=&\f{\hbar^2}{8}\langle u^*\varrho+(\rho^*+\varrho)\omega,\ti(\hat{B}\ti^{k+1}\varrho\cdot\ti^{k+1}\varrho)\rangle\\
\leq&C\{(\|\omega\|_{L^3}+\|u^*\|_{L^6}\|\varrho\|_{L^6})\|\ti^{k+2}\varrho\|\|\ti^{k+1}\varrho\|_{L^6}\\
&+\|(\varrho,\omega)\|_{L^6}\|(\ti \rho^*,\ti \varrho)\|\|\ti^{k+1}\varrho\|_{L^6}^2\},\\
\leq&C\delta\|\ti \varrho\|^2_{k+1},
\end{align*}
where
\begin{align*}
\hat{B}=\f{1}{P'(\rho^*+\varrho)}\bigg[1-\f{P''(\rho^*+\varrho)}{P'(\rho^*+\varrho)}(\rho^*+\varrho)\bigg].
\end{align*}
For $I_{17}$, we have
\begin{align*}
I_{17}\leq& C\sum_{l<k}C_k^l\|\ti^{k+2}\varrho\|\|\hat{A}\|_{L^6}(\|\ti^{k+1-l}(\rho^*+\varrho)\|_{L^6}\|\ti^l \omega\|_{L^6}+\|\ti^{k-l}(\rho^*+\varrho)\|_{L^6}\|\ti^{l+1}\omega\|_{L^6}),\\
\leq& C\delta\|\ti(\varrho,\omega)\|_{k+1,k}^2.
\end{align*}
The other terms $I_{14}$-$I_{16}$ and $I_{18}$ can be obtained similarly, we have
\begin{align*}
I_{1}\leq C\delta\|\ti(\varrho,\omega)\|_{k+1,k}^2.
\end{align*}
For the term $I_7$, we have
\begin{align*}
I_7\leq&\bigg|\bigg\langle\f{\varrho(t)}{\rho^*(\varrho+\rho^*)}\ti^k(\mu\Delta\omega+(\mu+\lambda)\ti\Div \omega),\tilde{A}\ti^k \omega\bigg\rangle\bigg|\\
&+|\langle F_k,\tilde{A}\ti^k \omega\rangle|+ \f{\hbar^2}{4}|\langle\ti^k J,\tilde{A}\ti^k \omega\rangle|\\
&=I_{71}+I_{72}+I_{73}.
\end{align*}
For $I_{71}$, by integration by parts, we obtain
\begin{align*}
I_{71}\leq&\bigg|\bigg\langle\ti\bigg(\f{\varrho}{\rho^*(\varrho+\rho^*)}\bigg)\ti^{k-1}(\mu\Delta\omega+(\mu+\lambda)\ti\Div \omega),\tilde{A}\ti^k \omega\bigg\rangle\bigg|\\
&+\bigg|\bigg\langle\f{\varrho}{\rho^*(\varrho+\rho^*)}\ti^{k-1}(\mu\Delta\omega+(\mu+\lambda)\ti\Div \omega),\ti(\tilde{A}\ti^k \omega)\bigg\rangle\bigg|\\
\leq&\|(\ti\varrho,\ti \rho^* )\|_{L^3}\|\ti^{k+1}\omega\|\|\ti^{k}\omega\|_{L^6}+\|\varrho\|_{L^\infty}\|\ti^{k+1}\omega\|^2\\
\leq&\delta\|\ti^{k+1}\omega\|^2.
\end{align*}
For $I_{72}$, if $k=0$, we have
\begin{align*}
I_{72}\leq&C\bigg\{\|(1+|x|)^2\ti u^*\|_{L^\infty}\bigg\|\f{\omega}{|x|}\bigg\|\bigg\|\f{\omega}{|x|}\bigg\|+\|(1+|x|)u^*\|_{L^\infty}\bigg\|\f{\omega}{|x|}\bigg\|\|\ti\omega\|\\
&+\|\omega\|_{L^3}\|\omega\|_{L^6}\|\ti\omega\|_{L^2}+\|(1+|x|)^2\ti \rho^*\|_{L^\infty}\bigg\|\f{\omega}{|x|}\bigg\|\bigg\|\f{\varrho}{|x|}\bigg\|\\
&+\|(1+|x|)^2\ti^2 u^*\|_{L^\infty}\bigg\|\f{\omega}{|x|}\bigg\|\bigg\|\f{\varrho}{|x|}\bigg\|+\|(1+|x|)^2G\|_{L^\infty}\bigg(\bigg\|\f{\omega}{|x|}\bigg\|\bigg\|\f{\omega}{|x|}\bigg\|
+\bigg\|\f{\varrho}{|x|}\bigg\|\bigg\|\f{\omega}{|x|}\bigg\|\bigg)\bigg\}\\
\leq&C\delta\|\ti(\varrho, \omega)\|^2.
\end{align*}
and if $1\leq k \leq 3$,
\begin{align*}
I_{72}\leq&C\bigg\{\|\ti^{k+1}u^*\|_{L^3}\|\omega\|_{L^6}\|\ti^k\omega\|+\sum\limits_{\nu=1}^{k+1}\|\ti^\nu\omega\|\|\ti^k\omega\|\\
&+\sum\limits_{\nu=1}^{k+1}(\|\ti^\nu\rho^*\|+\|\ti^{\nu+1}u^*\|)\|\varrho\|_{L^6}\|\ti^k\omega\|_{L^3}+\sum\limits_{\nu=1}^{k}\|\ti^\nu\varrho\|\|\ti^k\omega\|\\
&+(\|\varrho\|_{L^6}+\|\omega\|_{L^6})\sum\limits_{\nu=0}^{k}\|\ti^\nu G\|_{L^3}\|\ti^k\omega\|+\|R_F^{k}\|_{L^\f{3}{2}}\|\ti^k\omega\|_{L^3}\bigg\}\\
&\leq C(\delta+\varepsilon)\|\ti(\varrho, \omega)\|_{k+1,k}^2+\varepsilon^{-1}\|\ti^k \omega\|^2.
\end{align*}
For $I_{73}$, recall \eqref{J}, we have
\begin{align}\label{J3}
I_{73}\leq& C(|\langle\nabla^k(\varrho\nabla\Delta\rho^*),\tilde{A}\nabla^k\omega\rangle|+|\langle\nabla^k((\varrho+\rho^*-1)\nabla\Delta\varrho),\tilde{A}\nabla^k\omega\rangle|
+|\langle\nabla^k(|\nabla\rho^*|^2\varrho),\tilde{A}\nabla^k\omega\rangle|\notag\\
&+|\langle\nabla^k(|\nabla\varrho|^2\nabla\rho^*),\tilde{A}\nabla^k\omega\rangle|
+|\langle\nabla^k(|\nabla\rho^*|^2\nabla\varrho),\tilde{A}\nabla^k\omega\rangle|
+|\langle\nabla^k(|\nabla\varrho|^2\nabla\varrho),\tilde{A}\nabla^k\omega\rangle|\notag\\
&+|\langle\nabla^k(\varrho\nabla\rho^*\ti^2\rho^*),\tilde{A}\nabla^k\omega\rangle|
+|\langle\nabla^k(\nabla\varrho\ti^2\rho^*),\tilde{A}\nabla^k\omega\rangle|
+|\langle\nabla^k(\nabla\varrho\ti^2\varrho),\tilde{A}\nabla^k\omega\rangle|\notag\\
&+|\langle\nabla^k(\nabla\rho^*\ti^2\varrho),\tilde{A}\nabla^k\omega\rangle|).
\end{align}
If $k=0$, we have
\begin{align*}
|\langle\varrho\nabla\Delta\rho^*,\tilde{A}\omega\rangle|&\leq\|\varrho\|_{L^6}\|\nabla\Delta\rho^*\|_{L^2}\|\tilde{A}\|_{L^6}\|\omega\|_{L^6}\\
&\leq C\delta(\|\ti \varrho\|^2+\|\ti \omega\|^2),
\end{align*}
and if $1\leq k \leq 3$,
\begin{align*}
&|\langle\nabla^k(\varrho\nabla\Delta\rho^*),\tilde{A}\nabla^k\omega\rangle|\\
&=|\langle\nabla^{k-1}(\varrho\nabla\Delta\rho^*),\ti(\tilde{A}\nabla^k\omega)\rangle|\\
&\leq C \sum_{0\leq l\leq k-1}\|\ti^l\varrho\|_{L^6}\|\ti^{k+2-l}\rho^*\|_{L^3}(\|\ti\tilde{A}\|_{L^\infty}\|\nabla^k\omega\|_{L^2}+\|\tilde{A}\|_{L^\infty}\|\nabla^{k+1}\omega\|_{L^2})\\
&\leq C\delta\|\ti(\varrho,\omega)\|_{k-1,k}^2.
\end{align*}
If $k=0$, we have
\begin{align*}
&|\langle(\varrho+\rho^*-1)\nabla\Delta\varrho,\tilde{A}\omega\rangle|\\
&\leq \|\ti(\varrho+\rho^*-1)\|_{L^\infty}\|\Delta\varrho\|_{L^2}\|\tilde{A}\|_{L^3}\|\omega\|_{L^6}\\
&~~+\|(\varrho+\rho^*-1)\|_{L^\infty}\|\Delta\varrho\|_{L^2}(\|\ti\tilde{A}\|_{L^3}\|\omega\|_{L^6}+\|\tilde{A}\|_{L^\infty}\|\ti\omega\|_{L^2})\\
&\leq C\delta(\|\ti^2 \varrho\|^2+\|\ti \omega\|^2),
\end{align*}
and if $1\leq k \leq 3$,
\begin{align*}
&|\langle\nabla^k((\varrho+\rho^*-1)\nabla\Delta\varrho),\tilde{A}\nabla^k\omega\rangle|\\
&=|\langle\nabla^{k-1}((\varrho+\rho^*-1)\nabla\Delta\varrho),\ti(\tilde{A}\nabla^k\omega)\rangle|\\
&\leq C \sum_{0\leq l\leq k-1}\|\ti^l(\varrho+\rho^*-1)\|_{L^\infty}\|\ti^{k+2-l}\varrho\|_{L^2}(\|\ti\tilde{A}\|_{L^\infty}\|\nabla^k\omega\|_{L^2}+\|\tilde{A}\|_{L^\infty}\|\nabla^{k+1}\omega\|_{L^2})\\
&\leq C\delta\|\ti(\varrho,\omega)\|_{k+1,k}^2.
\end{align*}
If $k=0$, we have
\begin{align*}
|\langle\nabla\varrho\Delta\varrho,\tilde{A}\omega\rangle|
&\leq\|\ti\varrho\|_{L^6}\|\Delta\varrho\|_{L^2}\|\tilde{A}\|_{L^6}\|\omega\|_{L^6}\\
&\leq C\delta(\|\ti^2 \varrho\|^2+\|\ti \omega\|^2),
\end{align*}
and if $1\leq k \leq 3$,
\begin{align*}
&|\langle\nabla^k(\nabla\varrho\Delta\varrho),\tilde{A}\nabla^k\omega\rangle|\\
&=|\langle\nabla^{k-1}(\nabla\varrho\Delta\varrho),\ti(\tilde{A}\nabla^k\omega)\rangle|\\
&\leq C \sum_{0\leq l\leq k-1}\|\ti^{l+1}\varrho\|_{L^3}\|\ti^{k+1-l}\varrho\|_{L^6}(\|\ti\tilde{A}\|_{L^\infty}\|\nabla^k\omega\|_{L^2}+\|\tilde{A}\|_{L^\infty}\|\nabla^{k+1}\omega\|_{L^2})\\
&\leq C\delta\|\ti(\varrho,\omega)\|_{k+1,k}^2.
\end{align*}
The other terms on the right hand side of \eqref{J3} can be estimated similarly. Then for $0\leq k \leq 3$, we have
\begin{align*}
&I_{73}\leq C\delta \|\ti(\varrho,\omega)\|_{k+1,k}^2.
\end{align*}
Collecting these terms, we get
\begin{align*}
I_{7}\leq C(\delta+\varepsilon)\|\ti(\varrho, \omega)\|_{k+1,k}^2+\varepsilon^{-1}\|\ti^k \omega\|^2.
\end{align*}
The terms $I_2-I_6$ and $I_8$ can also be bounded by $C\delta \|\ti(\varrho,\omega)\|_{k+1,k}^2$. Choosing $$d_1=\min\limits_{\bar{\rho}/2\leq s \leq 3\bar{\rho}/2}\f{4\mu s^2}{5\bar{\rho}P'(s)}.$$Owing to the estimates $I_1 \thicksim I_{8} $ and the smallness of $\delta$ and $\varepsilon$,, we conclude the proof of this lemma.
\end{proof}
For the estimates on $\ti \varrho$ and its derivatives up to $\ti^5 \varrho$, we have
\begin{lemma}\label{As3.3}
Let $(\varrho,\omega)\in \mathfrak{C}(0,t_1;\mathcal{H}^{4,3})$ be a solution to the initial value problem \eqref{A3.1}. Then there exist constants $\delta_0,\varepsilon_0>0$ and $d_2>0$ such that $0<\delta\leq \delta_0$, $0<\varepsilon\leq \varepsilon_0$, we have
 \begin{eqnarray}\label{As3.40}
\f{d}{dt}\langle\omega,\ti \varrho\rangle+d_2\|\ti\varrho\|^2+{\f{\hbar^2}{4}}\|\ti^2\varrho\|^2\leq C\|\ti \omega\|^2.
	\end{eqnarray}
Moreover, for $1\leq k \leq 3$, it holds that
\begin{eqnarray}\label{As3.41}
\f{d}{dt}\langle\ti^k\omega,\ti^{k+1} \varrho\rangle+d_2\|\ti^{k+1}\varrho\|^2+{\f{\hbar^2}{4}}\|\ti^{k+2}\varrho\|^2\leq C\|\ti (\varrho,\omega)\|^2_{k-1,k},
\end{eqnarray}
where the constant $C>0$ depends only on $\mu,\lambda,\hbar$.
\end{lemma}
\begin{proof}
Appling $\ti^k$ to $\eqref{A3.1}_2$ and multiplying the resultant identity by $\ti^{k+1}\varrho$, then integrating over $\mathbb{R}^3$, we have
\begin{align*}
&\langle A\ti^{k+1}\varrho,\ti^{k+1}\varrho\rangle+\f{\hbar^2}{4}\|\ti^{k+2}\varrho\|^2\\
&=-\f{d}{dt}\langle \ti^{k}\omega,\ti^{k+1}\varrho\rangle+\langle \ti^{k+1}\omega,\ti^{k+1}((\varrho+\rho^*)\omega)\rangle
+\langle \ti^{k+1}\omega,\ti^{k+1}(\varrho u^*)\rangle+\langle\ti^k f,\ti^{k+1} \varrho\rangle\\
&+\sum_{l<k}C_k^l\langle \ti^{k-l}A\ti^{l+1}\varrho,\ti^{k+1}\varrho\rangle
+\bigg\langle\ti^{k}\bigg(\f{1}{\rho^*}(\mu\Delta\omega+(\mu+\lambda)\ti\Div \omega)\bigg),\ti^{k+1} \varrho\bigg\rangle=\sum\limits_{i=1}^{6}J_i.
\end{align*}
Similar to the estimates for the terms in the proof of Lemma \ref{As3.2}, we have
\begin{align*}
&J_{4}\leq C(\delta+\varepsilon)\|\ti^{k+1}\varrho\|_{1}^2+\varepsilon^{-1}\|\ti (\varrho,\omega)\|_{k-1,k}^2,\\
&J_{6}\leq C\varepsilon\|\ti^{k+2}\varrho\|^2+\varepsilon^{-1}\|\ti \omega\|_{k}^2,\\
&J_{2},J_{3},J_{5}\leq C\delta\|\ti^{k+2}\varrho\|^2+C\delta\|\ti (\varrho,\omega)\|_{k,k}^2.
\end{align*}
Choosing $$d_2=\min\limits_{\bar{\rho}/2\leq s \leq 3\bar{\rho}/2}\f{P'(s)}{s}.$$Combing these estimates and the smallness of $\delta$ and $\varepsilon$, we complete the proof of lemma \ref{As3.3}.
\end{proof}
Now, we start to give the proof of Proposition \ref{Pro3.2}. Let $(\varrho,\omega)\in \mathfrak{C}(0,t_1;\mathcal{H}^{4,3})$ be a solution to the problem \eqref{A3.1}. Furthermore, we suppose that $\|(\varrho,\omega)(t)\|_{4,3}+\|(\sigma^*,u^*)\|_{\mathcal{F}^{4,5}}\leq\delta$, where $\delta$ is small enough such that the results obtained in Lemma \ref{As3.2}-\ref{As3.3} hold true.

Define
\begin{eqnarray}\label{definition1}
[\varrho,\omega](t)=\|\varrho(t)\|^2+\langle\hat{A}(t)\ti \varrho(t),\ti \varrho(t)\rangle+\langle\tilde{A}(t)\omega(t),\omega(t)\rangle,
	\end{eqnarray}
where $\hat{A}(t)$, $\tilde{A}(t)$ are defined as in Lemma \ref{As3.2}.

Multiplying \eqref{As3.40} by a small constant ${\lambda_0}$, then summing up the resultant equation and \eqref{As3.10} implies that
\begin{eqnarray}\label{As3.49}
\f{d}{dt}\{a_0[\varrho,\omega]+b_0\langle\omega,\ti \varrho\rangle\}+\|(\ti \varrho,\ti \omega)\|_{1,0}^2\leq 0,
	\end{eqnarray}
where and hereafter $a_\nu,b_\nu>0$, $\nu=0,1,\cdots,3$ are constants depending only on $\mu,\lambda,\hbar$.

Summing up \eqref{As3.41}, \eqref{As3.11} and \eqref{As3.49} with $k=1$ yields
\begin{eqnarray}\label{As3.50}
\f{d}{dt}\bigg\{\sum\limits_{\nu=0}^{1}a_\nu[\ti^\nu\varrho,\ti^\nu\omega]+\sum\limits_{\nu=0}^{1}b_\nu\langle\ti^\nu\omega,\ti^{\nu+1} \varrho\rangle\bigg\}+\|(\ti \varrho,\ti \omega)\|_{2,1}^2\leq 0.
	\end{eqnarray}
Similarly, summing up \eqref{As3.41}, \eqref{As3.11} and \eqref{As3.50} with $k=2$ gives
\begin{eqnarray}\label{As3.51}
\f{d}{dt}\bigg\{\sum\limits_{\nu=0}^{2}a_\nu[\ti^\nu\varrho,\ti^\nu\omega]+\sum\limits_{\nu=0}^{2}b_\nu\langle\ti^\nu\omega,\ti^{\nu+1} \varrho\rangle\bigg\}+\|(\ti \varrho,\ti \omega)\|_{3,2}^2\leq 0.
	\end{eqnarray}
Also, when $k=3$, summing up \eqref{As3.41}, \eqref{As3.11} and \eqref{As3.51}, we obtain
\begin{eqnarray}\label{As3.52}
\f{d}{dt}\bigg\{\sum\limits_{\nu=0}^{3}a_\nu[\ti^\nu\varrho,\ti^\nu\omega]+\sum\limits_{\nu=0}^{3}b_\nu\langle\ti^\nu\omega,\ti^{\nu+1} \varrho\rangle\bigg\}+\|(\ti \varrho,\ti \omega)\|_{4,3}^2\leq 0.
	\end{eqnarray}
Then integrating \eqref{As3.52} over $[0,t]$ implies that
\begin{eqnarray}\label{As3.53}
N[\varrho,\omega](t)+\int\limits_0^t\|(\ti \varrho,\ti \omega)(s)\|_{4,3}^2ds\leq N[\varrho,\omega](0),
	\end{eqnarray}
for any $t\in[0,t_1]$, where
\begin{eqnarray*}
N[\varrho,\omega](t)=\sum\limits_{\nu=0}^{3}a_\nu[\ti^\nu\varrho,\ti^\nu\omega](t)+\sum\limits_{\nu=0}^{3}b_\nu\langle\ti^\nu\omega(t),\ti^{\nu+1} \varrho(t)\rangle,~~t>0.
	\end{eqnarray*}
Set
\begin{eqnarray*}
B_0=\min\limits_{\bar{\rho}/2\leq s \leq 3\bar{\rho}/2}\{\tilde{A}(s),1\},~~B_1=\max\limits_{\bar{\rho}/2\leq s \leq 3\bar{\rho}/2}\{\tilde{A}(s),1\}.
	\end{eqnarray*}
Without loss of generality, we assume that $a_\nu\leq a_{\nu-1}$ and $b_\nu\leq a_\nu {\min\{B_0,1\}}/{4}$, we have
\begin{eqnarray}\label{As3.54}
\f{a_3}{4}B_0\|( \varrho, \omega)(s)\|_{4,3}^2\leq N[\varrho,\omega](t)\leq 2a_0B_1\|( \varrho, \omega)(s)\|_{4,3}^2,
	\end{eqnarray}
for any $t\in[0,t_1]$.

Combing \eqref{As3.53} with \eqref{As3.54}, we finish the proof of Proposition \ref{Pro3.2}. Hence, by a standard continuity argument, this closes the a priori estimates $\|(\varrho,\omega)(t)\|_{4,3}\leq\delta$ if we assume that $\|(\varrho,\omega)(0)\|_{4,3}\leq\delta$. The global existence can be achieved by a standard continuity argument combined with Proposition \ref{Pro3.1}. This completes the proof of Theorem \ref{thm1.2}.
\section{Decay for the nonlinear system}
This section is dedicated to prove Theorem \ref{thm1.3}. First, we obtain the energy estimates of the negative Sobolev norms of solutions. Then, we prove the decay rates of solutions around the steady state. We set $\varrho=\rho-\rho^*$, $\mathcal{M}=m-m^*$, the problem \eqref{1.1} can be rewritten as
\begin{equation}\label{4.1}
		\left\{ \begin{aligned}
			&\varrho_t+\mathrm{div} \mathcal{M}=0,\\
 			&\mathcal{M}_t-\mu\Delta \mathcal{M}-(\mu+\lambda)\ti\Div \mathcal{M}+\ti\varrho-\f{\hbar^2}{4}\ti\Delta\varrho=Q,\\
            &(\varrho,\mathcal{M})|_{t=0}=(\rho_0-\rho^*,\rho_0 u_0-\rho^*u^*).
		\end{aligned} \right.
	\end{equation}
	Here the nonlinear function $Q$ is defined as
	\begin{eqnarray*}
Q&=&F\varrho-\Div\bigg( \f{(\mathcal{M}+m^*)\otimes(\mathcal{M}+{m}^*)}{\varrho+\rho^*}-\f{m^*\otimes{m}^*}{\rho^*}\bigg)\\
        &&-\mu\Delta\bigg({\mathcal{M}}-\f{\mathcal{M}}{\varrho+\rho^*}\bigg)-\mu\Delta\bigg(\f{m^*}{\rho^*}-\f{m^*}{\varrho+\rho^*}\bigg)\\
        &&-(\mu+\lambda)\ti\Div\bigg({\mathcal{M}}-\f{\mathcal{M}}{\varrho+\rho^*}\bigg)-(\mu+\lambda)\ti\Div\bigg(\f{m^*}{\rho^*}-\f{m^*}{\varrho+\rho^*}\bigg)\\
        &&-(P'(\varrho+\rho^*)-1)\ti\varrho-(P'(\varrho+\rho^*)-P'(\rho^*)\ti\rho^*\\
		&&+\f{\hbar^2}4\bigg(\bigg(\f{|\ti (\varrho+\rho^*)|^2}{(\varrho+\rho^*)^2}-\f{|\ti {\rho^*}|^2}{(\rho^*)^2}\bigg)\ti {\rho^*}+\f{|\ti (\varrho+\rho^*)|^2}{(\varrho+\rho^*)^2}\ti \varrho\\
		&&~~~~~~~~~-\bigg(\f{\ti (\varrho+\rho^*) }{\varrho+\rho^*}-\f{\ti {\rho^*} }{\rho^*}\bigg)\Delta {\rho^*}-\f{\ti (\varrho+\rho^*) }{\varrho+\rho^*}\Delta \varrho\\
		&&~~~~~~~~~-\bigg(\f{\ti(\varrho+\rho^*)}{\varrho+\rho^*}-\f{\ti {\rho^*} }{\rho^*}\bigg)\ti^2{\rho^*}-\f{\ti (\varrho+\rho^*) }{\varrho+\rho^*}\ti^2\varrho\bigg).
	\end{eqnarray*}
By using the mean value theorem, we have
\begin{eqnarray}\label{Q}
Q&\sim& F\varrho+\mathcal{M}^2\ti(\varrho+\rho^*)+\mathcal{M}\ti \mathcal{M}+m^*\ti \mathcal{M}+\mathcal{M} \ti m^*+m^*\mathcal{M}\ti(\varrho+\rho^*)\notag\\
&&+\varrho m^*\ti m^*+(m^*)^2\ti\varrho+(\rho^*-1+\varrho)\ti \varrho+\varrho\ti \rho^*+(\rho^*-1+\varrho)\ti^2\mathcal{M}\notag\\
&&+\ti^2(\rho^*-1+\varrho)\mathcal{M}+\ti(\rho^*-1+\varrho)\ti\mathcal{M}+m^*\ti^2\varrho+\ti m^*\ti\varrho+\ti^2m^*\varrho\notag\\
&&+|\ti\rho^*|^2\ti \rho^*\varrho+\ti \rho^*\ti \varrho\ti( \varrho+\rho^*)+|\ti( \varrho+\rho^*)|^2\ti\varrho+\ti \rho^*\varrho\Delta\rho^*\notag\\
&&+\ti \varrho\Delta\rho^*+\Delta\varrho\ti( \varrho+\rho^*)+\ti \rho^*\varrho\ti^2\rho^*+\ti \varrho\ti^2\rho^*+\ti^2\varrho\ti( \varrho+\rho^*).
	\end{eqnarray}
First, we will give the negative Sobolev estimates for the problem \eqref{4.1}.
\begin{lemma}\label{lma5.1}
For $s\in(0,\f32)$, it holds that
\begin{eqnarray}\label{ff37}
&&\f{d}{dt}\int\rho^*(|\Lambda^{-s}\varrho|^2+\f{\hbar^2}{4}| \Lambda^{-s}\ti \varrho|^2+|\Lambda^{-s}\mathcal{M}|^2)dx+C\|\Lambda^{-s}\ti\mathcal{M}\|^2\notag\\
&&~~\leq C_1\delta\| \Lambda^{-s}\ti\varrho\|^2+C\delta(\|\varrho\|_{H^2}^2+\|\ti \mathcal{M}\|_{H^1}^2).
\end{eqnarray}
\end{lemma}
\begin{proof}
Applying $\Lambda^{-s}$ to $\eqref{4.1}_1,~\eqref{4.1}_2$ and multiplying the resulting identities by $\rho^*\Lambda^{-s}\varrho$, $\rho^*\Lambda^{-s}\mathcal{M}$ respectively, summing up them and then integrating the resultants over $\mathbb{R}^3$ by parts, we have
\begin{align}\label{ff38}
&\f12\f{d}{dt}\int\rho^*(|\Lambda^{-s}\varrho|^2+|\Lambda^{-s}\mathcal{M}|^2)dx+\int\rho^*(\mu|\ti\Lambda^{-s}\mathcal{M}|^2+(\mu+\lambda)|\Div\Lambda^{-s}\mathcal{M}|^2)dx\notag\\
&~~-\f{\hbar^2}{4}\int\rho^*\Lambda^{-s}\ti\Delta\varrho\Lambda^{-s}Mdx\notag\\
&=\int\Lambda^{-s}\varrho\Lambda^{-s}\mathcal{M}\ti\rho^*dx-\int(\mu\ti\Lambda^{-s}\mathcal{M}+(\mu+\lambda)\Div\Lambda^{-s}\mathcal{M})\Lambda^{-s}\mathcal{M}\ti\rho^*dx\notag\\
&~~+\int\rho^*\Lambda^{-s}Q\Lambda^{-s}\mathcal{M}dx.
\end{align}
For the last term on the left hand side of \eqref{ff38}, we obtain
\begin{align}\label{ff39}
&-\f{\hbar^2}{4}\int\rho^*\Lambda^{-s}\ti\Delta\varrho\Lambda^{-s}Mdx\notag\\
&=\f{\hbar^2}{4}\int\Lambda^{-s}\Delta\varrho\Div(\rho^*\Lambda^{-s}M)dx\notag\\
&=\f{\hbar^2}{4}\int(\Lambda^{-s}\Delta\varrho\ti\rho^*\Lambda^{-s}M
-\Lambda^{-s}\ti\varrho\ti\rho^*\Lambda^{-s}\Div M
+\rho^*\Lambda^{-s}\ti\varrho\Lambda^{-s}\ti \varrho_t)dx\notag\\
&=\f{\hbar^2}{8}\f{d}{dt}\int\rho^*|\Lambda^{-s}\ti\varrho|^2dx
-\f{\hbar^2}{4}\int[\Lambda^{-s}\ti\varrho\otimes\ti\rho^*:\Lambda^{-s}\ti M\notag\\
&~~~~+\Lambda^{-s}\ti\varrho\ti^2\rho^*\Lambda^{-s} M
+\Lambda^{-s}\ti\varrho\ti\rho^*\Lambda^{-s}\Div M]dx\notag\\
&\geq\f{\hbar^2}{8}\f{d}{dt}\int\rho^*|\Lambda^{-s}\ti\varrho|^2dx
-C\|\Lambda^{-s}\ti\varrho\|\|\ti\rho^*\|_{L^\infty}\|\Lambda^{-s}\ti M\|\notag\\
&~~~~-C\|\Lambda^{-s}\ti\varrho\|\|\ti^2\rho^*\|_{L^3}\|\Lambda^{-s} M\|_{L^6}
-C\|\Lambda^{-s}\ti\varrho\|\|\ti\rho^*\|_{L^\infty}\|\Lambda^{-s}\Div M\|\notag\\
&\geq\f{\hbar^2}{8}\f{d}{dt}\int\rho^*|\Lambda^{-s}\ti\varrho|^2dx
-C\delta\|\Lambda^{-s}\ti\varrho\|^2-C\delta\|\Lambda^{-s}\ti M\|^2.
\end{align}
Now, we start to estimate the right hand side of \eqref{ff38}, the first term can be bounded as
\begin{align}\label{ff40}
\int\Lambda^{-s}\varrho\Lambda^{-s}\mathcal{M}\ti\rho^*dx&\leq C\bigg\|\f{\Lambda^{-s}\varrho}{1+|x|}\bigg\|\|\Lambda^{-s}\mathcal{M}\|_{L^6}\|(1+|x|)\ti\rho^*\|_{L^3}\notag\\
&\leq C\delta(\|\Lambda^{-s}\ti\varrho\|^2+\|\Lambda^{-s}\ti\mathcal{M}\|^2).
\end{align}
For the second term, we have
\begin{align}\label{ff41}
\int(\mu\ti\Lambda^{-s}\mathcal{M}+(\mu+\lambda)\Div\Lambda^{-s}\mathcal{M})\Lambda^{-s}\mathcal{M}\ti\rho^*dx&\leq C\|\ti\Lambda^{-s}\mathcal{M}\|\|\Lambda^{-s}\mathcal{M}\|_{L^6}\|\ti\rho^*\|_{L^3}\notag\\
&\leq C\delta\|\Lambda^{-s}\ti\mathcal{M}\|^2.
\end{align}
For the last term, we have
\begin{align*}
\int\rho^*\Lambda^{-s}Q\Lambda^{-s}\mathcal{M}dx&\leq C\|\rho^*\|_{L^3}\|\Lambda^{-s}Q\|\|\Lambda^{-s}\mathcal{M}\|_{L^6}\\
&\leq C\|\Lambda^{-s}Q\|\|\Lambda^{-s}\ti\mathcal{M}\|.
\end{align*}
In order to estimate $\|\Lambda^{-s}Q\|$, recall \eqref{Q}, we have
\begin{eqnarray}\label{QQ}
\|\Lambda^{-s}Q\|&\leq& \|\Lambda^{-s}(F\varrho)\|+\|\Lambda^{-s}(\mathcal{M}^2\ti(\varrho+\rho^*))\|+\|\Lambda^{-s}(\mathcal{M}\ti \mathcal{M})\|+\|\Lambda^{-s}(m^*\ti \mathcal{M})\|\notag\\
&&+\|\Lambda^{-s}(\mathcal{M} \ti m^*)\|+\|\Lambda^{-s}(m^*\mathcal{M}\ti(\varrho+\rho^*))\|+\|\Lambda^{-s}(\varrho m^*\ti m^*)\|+\|\Lambda^{-s}((m^*)^2\ti\varrho)\|\notag\\
&&+\|\Lambda^{-s}((\rho^*-1+\varrho)\ti \varrho)\|+\|\Lambda^{-s}(\varrho\ti \rho^*)\|+\|\Lambda^{-s}((\rho^*-1+\varrho)\ti^2\mathcal{M})\|\notag\\
&&+\|\Lambda^{-s}(\ti^2(\rho^*-1+\varrho)\mathcal{M})\|+\|\Lambda^{-s}(\ti(\rho^*-1+\varrho)\ti\mathcal{M})\|+\|\Lambda^{-s}(m^*\ti^2\varrho)\|\notag\\
&&+\|\Lambda^{-s}(\ti m^*\ti\varrho)\|+\|\Lambda^{-s}(\ti^2m^*\varrho)\|+\|\Lambda^{-s}(|\ti\rho^*|^2\ti \rho^*\varrho)\|\notag\\
&&+\|\Lambda^{-s}(\ti \rho^*\ti \varrho\ti( \varrho+\rho^*))\|+\|\Lambda^{-s}(|\ti( \varrho+\rho^*)|^2\ti\varrho)\|+\|\Lambda^{-s}(\ti \rho^*\varrho\ti^2\rho^*)\|\notag\\
&&+\|\Lambda^{-s}(\ti \varrho\ti^2\rho^*)\|+\|\Lambda^{-s}(\ti^2\varrho\ti( \varrho+\rho^*))\|.
	\end{eqnarray}
If $s\in(0,\f12]$, it holds from Lemma \ref{fflma2.1}, Lemma \ref{Ala3.3}, H\"older and Young inequalities that
\begin{align*}
\|\Lambda^{-s}(|\ti\rho^*|^2\ti \rho^*\varrho)\|&\leq C\||\ti\rho^*|^2\ti \rho^*\varrho\|_{L^{\f{1}{\f{1}{2}+\f{s}{3}}}}\\
&\leq C\|\ti\rho^*\|_{L^\infty}\|\ti\rho^*\|\|\ti^2 \rho^*\|^{\f12+s}\|\ti^3 \rho^*\|^{\f12-s}\|\varrho\|_{L^\infty}\\
&\leq C\delta\|\ti\varrho\|_{1}.
\end{align*}
Similarly, we have
\begin{align*}
\|\Lambda^{-s}(\ti^2\varrho\ti( \varrho+\rho^*))\|&\leq C\|\ti^2\varrho\ti( \varrho+\rho^*)\|_{L^{\f{1}{\f{1}{2}+\f{s}{3}}}}\\
&\leq C\|\ti^2\varrho\|\|\ti^2 ( \varrho+\rho^*)\|^{\f12+s}\|\ti^3 ( \varrho+\rho^*)\|^{\f12-s}\\
&\leq C\delta\|\ti^2\varrho\|,
\end{align*}
and
\begin{align*}
\|\Lambda^{-s}(m^*\mathcal{M}\ti( \varrho+\rho^*))\|&\leq C\|m^*\mathcal{M}\ti( \varrho+\rho^*)\|_{L^{\f{1}{\f{1}{2}+\f{s}{3}}}}\\
&\leq C\|\ti ( \varrho+\rho^*)\|\|\ti m^*\|^{\f12+s}\|\ti^2 m^*\|^{\f12-s}\|\mathcal{M}\|_{L^\infty}\\
&\leq C\delta\|\ti\mathcal{M}\|_{1},
\end{align*}
where
\begin{align*}
&\|\ti m^*\|\leq\|\rho^*\|_{L^\infty}\|\ti u^*\|+\|\ti\rho^*\|\| u^*\|_{L^\infty}\leq C\delta,\\
&\|\ti^2 m^*\|\leq\|\rho^*\|_{L^\infty}\|\ti^2 u^*\|+\|\ti^2\rho^*\|\| u^*\|_{L^\infty}+\|\ti\rho^*\|_{L^3}\| \ti u^*\|_{L^6}\leq C\delta.
\end{align*}
The remaining terms on the right hand side of \eqref{QQ} can be bounded in a similar way. Hence, we have
\begin{align*}
\|\Lambda^{-s}Q\|\leq C\delta(\|\varrho\|_{2}+\|\ti \mathcal{M}\|_{1}).
\end{align*}
Now if $s\in(\f12, \f32)$, since $1/2+s/3<1$ and $2<3/s<6$, we shall estimate above terms between $L^2$ and $L^6$. For the terms on the right hand side of \eqref{QQ}, by using the integration by parts, Lemma \ref{fflma2.1}, Lemma \ref{Ala3.3}, Hardy, H\"older and Young inequalities, we have
\begin{align*}
\|\Lambda^{-s}(|\ti\rho^*|^2\ti \rho^*\varrho)\|&\leq C\||\ti\rho^*|^2\ti \rho^*\varrho\|_{L^{\f{1}{\f{1}{2}+\f{s}{3}}}}\\
&\leq C\|\ti\rho^*\|_{L^\infty}\|\ti\rho^*\|\|\ti \rho^*\|^{s-\f12}\|\ti^2 \rho^*\|^{\f32-s}\|\varrho\|_{L^\infty}\\
&\leq C\delta\|\ti\varrho\|_{1}.
\end{align*}
Similarly, we have
\begin{align*}
\|\Lambda^{-s}(\ti^2\varrho\ti( \varrho+\rho^*))\|&\leq C\|\ti^2\varrho\ti( \varrho+\rho^*)\|_{L^{\f{1}{\f{1}{2}+\f{s}{3}}}}\\
&\leq C\|\ti^2\varrho\|\|\ti ( \varrho+\rho^*)\|^{s-\f12}\|\ti^2 ( \varrho+\rho^*)\|^{\f32-s}\\
&\leq C\delta\|\ti^2\varrho\|,
\end{align*}
and
\begin{align*}
\|\Lambda^{-s}(m^*\mathcal{M}\ti( \varrho+\rho^*))\|&\leq C\|m^*\mathcal{M}\ti( \varrho+\rho^*)\|_{L^{\f{1}{\f{1}{2}+\f{s}{3}}}}\\
&\leq C\|\ti ( \varrho+\rho^*)\|\| m^*\|^{s-\f12}\|\ti m^*\|^{\f32-s}\|\mathcal{M}\|_{L^\infty}\\
&\leq C\delta\|\ti\mathcal{M}\|_{1},
\end{align*}
where
\begin{align*}
&\| m^*\|\leq\bigg\|\f{\rho^*}{1+|x|}\bigg\|_{L^\infty}\|(1+|x|) u^*\|\leq C\delta.
\end{align*}
For the fifth term on the right hand side of \eqref{QQ}, we have
\begin{align*}
\|\Lambda^{-s}(\mathcal{M} \ti m^*)\|&\leq C\|\mathcal{M} \ti m^*\|_{L^{\f{1}{\f{1}{2}+\f{s}{3}}}}\\
&\leq C\|(1+|x|)\ti m^*\|\bigg\| \f{\mathcal{M}}{1+|x|}\bigg\|^{s-\f12}\bigg\|\ti\f{\mathcal{M}}{1+|x|}\bigg\|^{\f32-s}\\
&\leq C\delta\|\ti\mathcal{M}\|_{1},
\end{align*}
where
\begin{align*}
&\|(1+|x|)\ti m^*\|\leq\bigg\|\f{\rho^*}{1+|x|}\bigg\|\|(1+|x|)^2\ti u^*\|_{L^\infty}+\|\ti\rho^*\|\| (1+|x|)u^*\|_{L^\infty}\leq C\delta.
\end{align*}
Similarly, the remaining terms on the right hand side of \eqref{QQ} can be bounded as
\begin{align*}
\|\Lambda^{-s}Q\|\leq C\delta(\|\varrho\|_{2}+\|\ti \mathcal{M}\|_{1}).
\end{align*}
Combing above estimates, we deduce
\begin{align}\label{ff42}
\int\rho^*\Lambda^{-s}Q\Lambda^{-s}\mathcal{M}dx&\leq C\|\Lambda^{-s}Q\|\|\ti\Lambda^{-s}\mathcal{M}\|\notag\\
&\leq C\delta\|\Lambda^{-s}\ti\mathcal{M}\|^2+C\delta(\|\varrho\|_{2}^2+\|\ti \mathcal{M}\|_{1}^2).
\end{align}
Substituting \eqref{ff39}-\eqref{ff41} and \eqref{ff42} into \eqref{ff38}, since $\delta$ is small enough, we deduce \eqref{ff37} and complete the proof of Lemma \ref{lma5.1}.
\end{proof}
\begin{lemma}\label{lma5.2}
For $s\in(0,\f32)$, it holds that
\begin{eqnarray}\label{ff43}
&&\f{d}{dt}\int\rho^*\Lambda^{-s}\mathcal{M}\Lambda^{-s}\ti \varrho dx+C\|\Lambda^{-s}\ti\varrho\|^2+C\int\rho^*|\Lambda^{-s}\Delta\varrho|^2dx\notag\\
&&~~\leq C_2\delta\| \Lambda^{-s}\ti\mathcal{M}\|^2+C\delta(\|\varrho\|_{H^2}^2+\|\ti \mathcal{M}\|_{H^1}^2).
\end{eqnarray}
\end{lemma}
\begin{proof}
Applying $\Lambda^{-s}$ to $\eqref{4.1}_2$ and multiplying the resulting identities by $\rho^*\Lambda^{-s}\ti\varrho$, we have
\begin{align}\label{ff44}
&\f12\f{d}{dt}\int\rho^*\Lambda^{-s}\mathcal{M}\Lambda^{-s}\ti \varrho dx+\int\rho^*|\Lambda^{-s}\ti \varrho|^2dx-\f{\hbar^2}{4}\int\rho^*\Lambda^{-s}\ti\Delta\varrho\Lambda^{-s}\ti \varrho dx\notag\\
&=\int\rho^*\Lambda^{-s}\mathcal{M}\Lambda^{-s}\ti \varrho_t dx+\int\rho^*(\mu\Lambda^{-s}\Delta\mathcal{M}+(\mu+\lambda)\Lambda^{-s}\ti\Div\mathcal{M})\Lambda^{-s}\ti \varrho dx\notag\\
&~~+\int\rho^*\Lambda^{-s}Q\Lambda^{-s}\ti \varrho dx.
\end{align}
For the last term on the left hand side of \eqref{ff44}, we have
\begin{align*}
&-\f{\hbar^2}{4}\int\rho^*\Lambda^{-s}\ti\Delta\varrho\Lambda^{-s}\ti \varrho dx\\
&=\f{\hbar^2}{4}\int(\rho^*|\Lambda^{-s}\Delta\varrho|^2dx+ \Lambda^{-s}\Delta\varrho\ti \rho^* \Lambda^{-s}\ti\varrho) dx\\
&\geq\f{\hbar^2}{4}\int\rho^*|\Lambda^{-s}\Delta\varrho|^2dx- C\|\sqrt{\rho^*}\Lambda^{-s}\Delta\varrho\|\bigg\|\f{\ti \rho^*}{\sqrt{\rho^*}}\bigg\|_{L^\infty} \|\Lambda^{-s}\ti\varrho\|\\
&\geq\f{\hbar^2}{4}\int\rho^*|\Lambda^{-s}\Delta\varrho|^2dx- C\delta\int\rho^*|\Lambda^{-s}\Delta\varrho|^2dx- C\delta\|\Lambda^{-s}\ti\varrho\|^2.
\end{align*}
For the terms on the right hand side of \eqref{ff44}, by $\eqref{4.1}_1$, Lemma \ref{fflma2.1}, Lemma \ref{Ala3.3}, H\"older and Young inequalities, we have
\begin{align*}
\int\rho^*\Lambda^{-s}\mathcal{M}\Lambda^{-s}\ti \varrho_t dx&=\int\Div(\rho^*\Lambda^{-s}\mathcal{M})\Lambda^{-s} \Div \mathcal{M}dx\\
&\leq C\delta\|\Lambda^{-s}\ti\mathcal{M}\|^2,
\end{align*}
and
\begin{align*}
&\int\rho^*(\mu\Lambda^{-s}\Delta\mathcal{M}+(\mu+\lambda)\Lambda^{-s}\ti\Div\mathcal{M})\Lambda^{-s}\ti \varrho dx\\
&=-\int\mu\Lambda^{-s}\ti\mathcal{M}\ti(\rho^*\Lambda^{-s}\ti \varrho)+(\mu+\lambda)\Lambda^{-s}\Div\mathcal{M}\Div(\rho^*\Lambda^{-s}\ti \varrho) dx\\
&\leq C\delta\|\Lambda^{-s}\ti\mathcal{M}\|^2+C\delta\int\rho^*|\Lambda^{-s}\Delta\varrho|^2dx+C\delta\|\Lambda^{-s}\ti\varrho\|^2.
\end{align*}
Finally, the estimate $\|\Lambda^{-s}Q\|$ in Lemma \ref{lma5.1} implies that
\begin{align*}
\int\rho^*\Lambda^{-s}Q\Lambda^{-s}\ti \varrho dx&\leq C\delta\|\Lambda^{-s}Q\|\|\Lambda^{-s}\ti\varrho\|\notag\\
&\leq C\delta\|\Lambda^{-s}\ti\varrho\|^2+C\delta(\|\varrho\|_{2}^2+\|\ti \mathcal{M}\|_{1}^2).
\end{align*}
The smallness of $\delta$ and above estimates yield \eqref{ff43}. This completes the proof of Lemma \ref{lma5.2}.
\end{proof}
$\mathbf{Proof~ of~ the~ estimate~ \eqref{1.3}.}$ If $s=0$, \eqref{1.3} holds obviously. Now if $s\in(0,\f32)$, multiplying \eqref{ff43} by a small constant $2C_1\delta$,
then summing up the resulting inequality and \eqref{ff37}, since $\delta$ is small enough, we have
\begin{eqnarray}\label{5.46}
&&\f{d}{dt}\mathcal{N}(t)+C(\|\Lambda^{-s}\ti\mathcal{M}\|^2+\|\Lambda^{-s}\ti\varrho\|^2)+C\int\rho^*|\Lambda^{-s}\Delta\varrho|^2dx\notag\\
&&\leq C\delta(\|\varrho\|_{H^2}^2+\|\ti \mathcal{M}\|_{H^1}^2),
\end{eqnarray}
where
\begin{eqnarray*}
\mathcal{N}(t)&=&\int\rho^*(|\Lambda^{-s}\varrho|^2+\f{\hbar^2}{4}| \Lambda^{-s}\ti \varrho|^2+|\Lambda^{-s}\mathcal{M}|^2+2C_1\delta\Lambda^{-s}\mathcal{M}\Lambda^{-s}\ti \varrho)dx\\
&\sim&\|\Lambda^{-s}(\varrho, \ti\varrho,\mathcal{M})(t)\|_{L^2}^2.
\end{eqnarray*}
By integrating \eqref{5.46} in time, we obtain
\begin{align*}
\|\Lambda^{-s}(\varrho, \ti\varrho,\mathcal{M})(t)\|_{L^2}^2&\leq \|\Lambda^{-s}(\varrho, \ti\varrho,\mathcal{M})(0)\|_{L^2}^2+C\delta\int_0^t(\|\varrho\|_{H^2}^2+\|\ti \mathcal{M}\|_{H^1}^2)d\tau\\
&\leq \|\Lambda^{-s}(\varrho, \ti\varrho,\mathcal{M})(0)\|_{L^2}^2+C\|( \varrho, \mathcal{M})(0)\|_{4,3}^2,
\end{align*}
which implies that \eqref{1.3} holds.~~~~~~~~~~~~~~~~~~~~~~~~~~~~~~~~~~~~~~~~~~~~~~~~~~~~~~~~~~~~~~~~~~~~~~~~~~~~~~~~~~~~~~~~~~~~~~~~~~~~~~~~~~~~~~~~$\square$

Now, we start to prove the time decay rates in Theorem \ref{thm1.3}. First, we will give a Lyapunov-type inequality.
\begin{lemma}\label{laB4.3}
Under the assumptions of Theorem \ref{thm1.3}, there exists a constant $C_1$ such that
\begin{eqnarray}\label{B3.8}
&&\f{d}{dt}\mathcal{E}(t)+C_1\|\ti({\varrho,\omega})\|_{3,2}^2\leq C\|\ti({\varrho,\omega})\|^2,
\end{eqnarray}
where $\mathcal{E}(t)$ is equivalent to $\|\ti({\varrho,\omega})\|_{3,2}^2$.
\end{lemma}
\begin{proof}
Recall the definition \eqref{definition1}, when $k=1$, multiplying \eqref{As3.41} by a small constant ${{\lambda_0}}$, then summing up the resultant equation and \eqref{As3.11} implies that
\begin{eqnarray}\label{Ass3.49}
\f{d}{dt}\{\alpha_1[\ti\varrho,\ti\omega]+\beta_1\langle\ti\omega,\ti^{2} \varrho\rangle\}+\|\ti^2( \varrho, \omega)\|_{1,0}^2\leq C\|\ti( \varrho,\omega)\|^2.
	\end{eqnarray}
Summing up \eqref{As3.41}, \eqref{As3.11} and \eqref{Ass3.49} with $k=2$ yields
\begin{eqnarray}\label{Ass3.51}
\f{d}{dt}\bigg\{\sum\limits_{\nu=1}^{2}\alpha_\nu[\ti^\nu\varrho,\ti^\nu\omega]+\sum\limits_{\nu=1}^{2}\beta_\nu\langle\ti^\nu\omega,\ti^{\nu+1} \varrho\rangle\bigg\}+\|\ti^2( \varrho,\omega)\|_{2,1}^2\leq C\|\ti( \varrho, \omega)\|^2.
	\end{eqnarray}
Similarly, summing up \eqref{As3.41}, \eqref{As3.11} and \eqref{Ass3.51} with $k=3$ gives
\begin{eqnarray}\label{Ass3.52}
\f{d}{dt}\mathcal{E}(t)+\|\ti^2( \varrho,\omega)\|_{3,2}^2\leq C\|\ti( \varrho, \omega)\|^2,
	\end{eqnarray}
where
\begin{eqnarray*}
\mathcal{E}(t)&=&\sum\limits_{\nu=1}^{3}\alpha_\nu[\ti^\nu\varrho,\ti^\nu\omega]+\sum\limits_{\nu=1}^{3}\beta_\nu\langle\ti^\nu\omega,\ti^{\nu+1} \varrho\rangle\\
&\sim&\|\ti({\varrho,\omega})\|_{3,2}^2.
\end{eqnarray*}
Adding $\|\ti( \varrho,\omega)\|^2$ to both hand sides of \eqref{Ass3.52}, we complete the proof of Lemma \ref{laB4.3}.
\end{proof}
Let ${U}=(\varrho,\mathcal{M})$,
\begin{equation}
 \nonumber
  \BF{A}=\left(
  \begin{matrix}
  0&{ \Div}\\
  \ti-\f{\hbar^2}{4}\ti\Delta &-\mu\Delta-(\mu+\lambda)\ti\Div
  \end{matrix}
  \right)
\end{equation}
and $E(t)$ be the semigroup generated by the linear operator ${A}$, then we have the following Lemma (see \cite{WT})
\begin{lemma}\label{la3.1}
Let $k\ge0$ be an integer and assume that $(\varrho,\mathcal{M})$ is the solution to the linearized system for equations \eqref{4.1} with the initial data $(\varrho_0,\mathcal{M}_0)$, then for any $1\leq p\leq 2$ and $q\geq2$, we have
\begin{eqnarray}
&&\|\varrho(t)\|_{L^q}\leq C(1+t)^{-\f{3}{2}(\f1p-\f1q)}(\|(\varrho_0,\mathcal{M}_0)\|_{L^p}+\|(\varrho_0,\mathcal{M}_0)\|_{L^q}),\notag\\
&&\|\ti^{k+1}\varrho(t)\|_{L^q}\leq C(1+t)^{-\f{3}{2}(\f1p-\f1q)-\f{k+1}{2}}(\|(\varrho_0,\mathcal{M}_0)\|_{L^p}+\|(\ti^{k+1}\varrho_0,\ti^k \mathcal{M}_0)\|_{L^q}),\label{3.3}\\
&&\|\ti^{k}\mathcal{M}(t)\|_{L^q}\leq C(1+t)^{-\f{3}{2}(\f1p-\f1q)-\f{k}{2}}(\|(\varrho_0,\mathcal{M}_0)\|_{L^p}+\|(\ti^{k+1}\varrho_0,\ti^k \mathcal{M}_0)\|_{L^q}).\notag
\end{eqnarray}
\end{lemma}
By the Duhamel's principle, the solution $(\varrho,\mathcal{M})$ can be written as
\begin{equation}\label{3.4}
	{(\varrho,\mathcal{M})}(t)={E}(t){(\varrho_0,\mathcal{M}_0)}+\int_0^t{E}(t-\tau)(0,Q)(\tau)d\tau.
\end{equation}
\begin{lemma}\label{la3.3}
Let $(\varrho,\mathcal{M})$ be the solution to the problem \eqref{4.1}, then under the assumptions of Theorem \ref{thm1.1}, for any $s\in[0,\f32)$, we have
\begin{eqnarray*}
  &&\|\ti({\varrho},\mathcal{M})(t)\|_{L^2}\leq CE_0(1+t)^{-\f{1+s}{2}}+C\delta\int_0^t(1+t-\tau)^{-\f{5}{4}}
  \|\ti ({\varrho},\mathcal{M})(\tau)\|_{3,2}d\tau,
\end{eqnarray*}
where $E_0=\|(\Lambda^{-s}\varrho_0,\Lambda^{-s}\mathcal{M}_0)\|_{ L^2}+\|(\varrho_0,\mathcal{M}_0)\|_{H^3}$.
\end{lemma}
\begin{proof} From \eqref{3.3} and \eqref{3.4}, we have
\begin{eqnarray}\label{A111}
\|\ti({\varrho},\mathcal{M})(t)\|_{L^2}\leq C\|\Lambda^{s}\ti E(t)\ast\Lambda^{-s}(\varrho_0,\mathcal{M}_0)\|_{L^2}+C\int_0^t(1+t-\tau)^{-\f{5}{4}}
\| ({0},Q)(\tau)\|_{L^1\cap H^1}d\tau.
\end{eqnarray}
By Lemma \ref{Ala3.2}, for $s\in[0,1]$, we have
\begin{eqnarray}
\|\Lambda^{s}\ti E(t)\ast\Lambda^{-s}(\varrho_0,\mathcal{M}_0)\|_{L^2}\leq C\|\ti E(t)\ast\Lambda^{-s}(\varrho_0,\mathcal{M}_0)\|_{L^2}^{1-s}\|\ti^2 E(t)\ast\Lambda^{-s}(\varrho_0,\mathcal{M}_0)\|_{L^2}^s,
\end{eqnarray}
for $s\in(1,\f32)$, we have
 \begin{eqnarray}
\|\Lambda^{s}\ti E(t)\ast\Lambda^{-s}(\varrho_0,\mathcal{M}_0)\|_{L^2}\leq C\|\ti^2 E(t)\ast\Lambda^{-s}(\varrho_0,\mathcal{M}_0)\|_{L^2}^{2-s}\|\ti^3 E(t)\ast\Lambda^{-s}(\varrho_0,\mathcal{M}_0)\|_{L^2}^{s-1}.
\end{eqnarray}
By the H\"{o}lder inequality, the Hardy inequality and Theorem \ref{thm1.2}, it holds that
\begin{eqnarray}\label{A112}
\| Q(\tau)\|_{L^1\cap H^1}\leq C\delta\|\ti ({\varrho},\mathcal{M})(\tau)\|_{3,2}.
\end{eqnarray}
By virtue of \eqref{A111}-\eqref{A112}, together with Lemma \ref{la3.1}, the proof is complete.
\end{proof}
Now, we define
\begin{eqnarray}
  N(t):=\sup\limits_{0\leq\tau\leq t}(1+\tau)^{1+s}\|\ti ({\varrho},\mathcal{M})(\tau)\|_{3,2}^2,\label{3.15}
\end{eqnarray}
then $N(t)$ is nondecreasing and satisfies
\begin{align*}
\|\ti ({\varrho},\mathcal{M})(\tau)\|_{3,2}\leq C(1+\tau)^{-\f{1+s}{2}}\sqrt{N(t)},~~0\leq \tau\leq t.
\end{align*}
From Lemma \ref{la3.2} and Lemma \ref{la3.3}, we have
\begin{eqnarray*}
\|\ti ({\varrho},\mathcal{M})(t)\|_{L^2}&\leq& CE_0(1+t)^{-\f{1+s}{2}}\\
  &&+C\delta\int_0^t(1+t-\tau)^{-\f54}
  (1+\tau)^{-\f{1+s}{2}}d\tau\sqrt{N(t)}\\
  &\leq& C(1+t)^{-\f{1+s}{2}}\bigg[E_0+
  \delta\sqrt{N(t)}\bigg].
\end{eqnarray*}
Then, by Gronwall's inequality, we deduce  from \eqref{B3.8}
\begin{eqnarray}
\mathcal{E}(t)&\leq& \mathcal{E}(0)e^{-C_1t}+C_2\int_0^te^{-C_1(t-\tau)}
  \|\ti ({\varrho},\mathcal{M})(\tau)\|^2_{L^2}d\tau\nt\\
  &\leq& \mathcal{E}(0)(1+t)^{-(1+s)}\nt\\
  &&+C\int_0^t(1+t-\tau)^{-(1+s)}
  (1+\tau)^{-(1+s)}d\tau
  \big[E_0^2+\delta^2{N(t)}\big]\nt\\
  &\leq&C(1+t)^{-(1+s)}\big[\mathcal{E}(0)+E_0^2+
  \delta^2{N(t)}\big],\label{3.16}
\end{eqnarray}
where $\mathcal{E}(t)$ is equivalent to $\|\ti ({\varrho},\omega)(t)\|_{3,2}^2$.

Recall $\mathcal{M}=m-m^*=(\varrho+\rho^*)\omega+\varrho u^*$, we have
\begin{eqnarray*}
 \|\ti \mathcal{M}\|_{L^2}&\leq&\|\ti(\varrho+\rho^*)\|_{L^3}\|\omega\|_{L^6}+\|(\varrho+\rho^*)\|_{L^\infty}\|\ti \omega\|_{L^2}+\|\ti u^*\|_{L^3}\|\varrho\|_{L^6}+\|u^*\|_{L^\infty}\|\ti \varrho\|_{L^2}\\
 &\leq& C\delta\|\ti({\varrho},\omega)\|_{L^2}.
\end{eqnarray*}
Similarly, the terms $ \|\ti^2 \mathcal{M}\|_{L^2}$ and $ \|\ti^3 \mathcal{M}\|_{L^2}$ can be bounded. Hence, $$\|\ti ({\varrho},\mathcal{M})(t)\|_{3,2}\leq C \mathcal{E}(t).$$
Then, we obtain from \eqref{3.15} and \eqref{3.16}
\begin{eqnarray*}
  N(t)\leq C(\mathcal{E}(0)+E_0^2+\delta^2N(t)).
\end{eqnarray*}
Choosing $\delta$ so small that $\delta\leq1/\sqrt{2C}$, we get
\begin{eqnarray*}
  N(t)\leq 2C(\mathcal{E}(0)+E_0^2),
\end{eqnarray*}
and hence from \eqref{3.16},
\begin{eqnarray}
  \|(\ti ({\varrho},\mathcal{M})(t)\|_{3,2}\leq C(1+t)^{-\f{1+s}{2}}.\label{*3.17}
\end{eqnarray}
Finally, by using Lemma \ref{la3.1} and \eqref{*3.17}, we have
\begin{eqnarray*}
\|({\varrho},\mathcal{M})(t)\|_{L^2}&\leq& CE_0(1+t)^{-\f{s}{2}}+C\delta\int_0^t(1+t-\tau)^{-\f{3}{4}}
  \|\ti ({\varrho},\mathcal{M})(\tau)\|_{H^2}d\tau\\
  &\leq&  CE_0(1+t)^{-\f{s}{2}}+C\delta\int_0^t(1+t-\tau)^{-\f{3}{4}}
 (1+\tau)^{-\f{1+s}{2}}d\tau\\
 &\leq&C(1+t)^{-\f{s}{2}}.
\end{eqnarray*}
This completes the proof of the time decay rates of Theorem \ref{thm1.3}.

\BF{Acknowledgments}
	
X. Xi was supported by NSFC (No. 11901127), Basic Research Project of Guangzhou, and Introduction of talent research start-up fund at Guangzhou University. The author will thank the referees for valuable comments and suggestions.

\end{document}